\documentclass[dvips,preprint]{imsart}
\RequirePackage[OT1]{fontenc} \RequirePackage{amsthm,amsmath}
\RequirePackage[numbers]{natbib}
\RequirePackage[colorlinks,citecolor=blue,urlcolor=blue]{hyperref}
\usepackage{amsmath}
\usepackage{amssymb}
\usepackage{amsmath}
\usepackage{amsfonts}
\usepackage[american]{babel}
\usepackage{graphicx}
\usepackage{flafter}
\usepackage[section]{placeins}

\startlocaldefs

\def\E{{\mathbb E }}

\def\goin{\to\infty}
\def\Bbb E{\mathbb{E}}
\def\Bbb R{\mathbb{R}}
\newtheorem{definition}{Definition}

\newtheorem{proposition}{Proposition}
\newtheorem{lemma}{Lemma}
\newtheorem{theorem}{Theorem}
\newtheorem{corollary}{Corollary}
\newtheorem{remark}{Remark}
\makeatletter \@addtoreset{equation}{section}

\makeatother

\font\tencmmib=cmmib10 \skewchar\tencmmib '60
\newfam\cmmibfam
\textfont\cmmibfam=\tencmmib

\font\tenmsb=msbm10 
\def\Bbb#1{\hbox{\tenmsb#1}}

\def\bbox{\quad\hbox{\vrule \vbox{\hrule \vskip2pt \hbox{\hskip2pt
\vbox{\hsize=1pt}\hskip2pt} \vskip2pt\hrule}\vrule}}
\def\lessim{\ \lower4pt\hbox{$
\buildrel{\displaystyle <}\over\sim$}\ }
\def\gessim{\ \lower4pt\hbox{$\buildrel{\displaystyle >}
\over\sim$}\ }

\def\goin{\to \infty}
\def\go0{\to 0}

\def\leftitem#1{\item{\hbox to\parindent{\enspace#1\hfill}}}

\def\qed{{$\hfill \bbox$}}

\def\sg{\sigma}

\def\sg2{\sigma^2}

\def\__{_{\infty}}

\numberwithin{equation}{section} \theoremstyle{plain}

\newcommand{\1}{{\rm 1}\kern-0.24em{\rm I}}
\def\E{\mathbb E}
\def\R{\mathbb R}
\newtheorem{assumption}{Assumption}

\newtheorem{rema}{Remark}

\endlocaldefs

\begin{document}

\begin{frontmatter}
\title{Asymptotics and Concentration Bounds for Bilinear Forms of Spectral Projectors of Sample Covariance} \runtitle{Asymptotics and concentration of spectral projectors}
\medskip

\textbf{This version is a part of the initial submission; another part was further developed and posted separately as Arxiv:1504.07333.}

\begin{aug}
\author{\fnms{Vladimir} \snm{Koltchinskii}\thanksref{t1}\ead[label=e1]{vlad@math.gatech.edu}} and 
\author{\fnms{Karim} \snm{Lounici}\thanksref{m1}\ead[label=e2]{klounici@math.gatech.edu}}
\thankstext{t1}{Supported in part by NSF Grants DMS-1207808, CCF-0808863 and CCF-1415498}
\thankstext{m1}{Supported in part by NSF CAREER Grant DMS-1454515 and Simons Collaboration Grant 315477}
\runauthor{V. Koltchinskii and K. Lounici}

\affiliation{Georgia Institute of Technology\thanksmark{m1}}

\address{School of Mathematics\\
Georgia Institute of Technology\\
Atlanta, GA 30332-0160\\
\printead{e1}\\
\printead*{e2}
}
\end{aug}

\begin{abstract}
Let $X,X_1,\dots, X_n$ be i.i.d. Gaussian random variables with zero mean
and covariance operator $\Sigma={\mathbb E}(X\otimes X)$ taking values in a separable Hilbert space ${\mathbb H}.$
Let 
$$
{\bf r}(\Sigma):=\frac{{\rm tr}(\Sigma)}{\|\Sigma\|_{\infty}}
$$ 
be the effective rank of $\Sigma,$
${\rm tr}(\Sigma)$ being the trace of $\Sigma$ and $\|\Sigma\|_{\infty}$ being
its operator norm. Let 
$$\hat \Sigma_n:=n^{-1}\sum_{j=1}^n (X_j\otimes X_j)$$ 
be the sample (empirical) covariance operator based on $(X_1,\dots, X_n).$  
The paper deals with a problem of estimation of spectral projectors of the covariance operator $\Sigma$ by their empirical counterparts, the spectral 
projectors of $\hat \Sigma_n$ (empirical spectral projectors). The focus is 
on the problems where both the sample size $n$ and the effective rank ${\bf r}(\Sigma)$ are large. This framework includes and generalizes well known high-dimensional spiked covariance models. Given a spectral projector $P_r$
corresponding to an eigenvalue $\mu_r$ of covariance operator $\Sigma$ and its 
empirical counterpart $\hat P_r,$ we derive sharp concentration bounds for  
bilinear forms of empirical spectral projector $\hat P_r$ in terms of sample size $n$ and effective 
dimension ${\bf r}(\Sigma).$
Building upon these concentration 
bounds, we prove the asymptotic normality of bilinear forms of random operators $\hat P_r -{\mathbb E}\hat P_r$ under the assumptions that $n\to \infty$ and 
${\bf r}(\Sigma)=o(n).$ In a special case of eigenvalues
of multiplicity one, these results are rephrased as concentration bounds and asymptotic normality for linear forms of empirical eigenvectors. 
Other results include bounds on the bias ${\mathbb E}\hat P_r-P_r$ and a method of bias reduction    
as well as a discussion of possible applications to statistical inference in high-dimensional principal component analysis.
\end{abstract}



\end{frontmatter}

\section{Introduction}\label{Sec:Intro}

{\it Principal Component Analysis (PCA)} is among the most popular methods of exploring  the covariance structure of a random process in a wide array of applications. It is of a particular interest in high-dimensional statistics as a tool of dimension reduction and feature extraction. 

Let $X$ be a random vector in $\R^p$ with zero mean and covariance matrix $\Sigma.$ The classical PCA is based on estimating the eigenvalues and the associated spectral projectors of $\Sigma$ by the eigenvalues and the spectral projectors of the sample covariance matrix $\hat \Sigma_n$ based on $n$ i.i.d. replications of $X,$ that is, the sample (empirical) eigenvalues and the sample (empirical) spectral projectors. Assessing the performance of the standard PCA raises naturally a question of how the sample eigenvalues and sample spectral projectors deviate from their population counterparts. In the 'standard setting', where $p\geq 1$ is fixed and $n\rightarrow \infty$, Anderson  \cite{Anderson} established the limiting joint distribution of the sample eigenvalues and the associated sample eigenvectors 
(see also Theorem 13.5.1 in \cite{Andersonbook}). These results have been extended
in \cite{DPR} to the case of i.i.d. data in infinite-dimensional Hilbert spaces (they have been used and further developed in numerous papers that followed, see, e.g, \cite{MR2014}).  

A number of authors considered a  'high-dimensional setting', where the dimension $p = p_n$ is allowed to grow with the sample size $n.$ Marchenko and Pastur \cite{Marcenko_Pastur} derived the ``limiting density'' of the spectrum of $\hat \Sigma_n$ in the case when $\Sigma = I_p$ is the identity matrix and $\frac{p}{n} \rightarrow c \in (0, 1]$ as $n\rightarrow \infty$ (more precisely, they obtained the a.s. limit of the empirical distribution of the eigenvalues). Under the same conditions, Johnstone \cite{Johnstone} proved that 
the largest empirical eigenvalue (properly normalized) converges in distribution to the Tracy-Widom law. The accuracy of this approximation was studied in \cite{Johnstone_Ma,Ma2012}. Assuming that the covariance matrix $\Sigma$ is 
the sum of the identity matrix and a small finite rank symmetric positive semi-definite perturbation, Baik, Ben Arrous and Peche \cite{Baik_Benarrous_Peche} discovered a phase transition effect where the sample versions of the non-unit eigenvalues satisfy different asymptotic properties that depend on how far from $1$ the non-unit eigenvalues are. Another line of research is a non-asymptotic theory  of sample covariance where the main goal is to obtain sharp non-asymptotic bounds
on the operator norm $\|\hat \Sigma_n-\Sigma\|_{\infty};$ 
a review of these results  can be found in \cite{vershynin}.

Concerning the estimation of spectral projectors, Johnstone and Lu \cite{Johnstone_Lu} proved that the classical PCA approach could fail to produce a consistent estimator when $\frac{p}{n} \rightarrow c >0$ as $n\rightarrow \infty.$ To overcome this difficulty, several authors proposed alternative estimators 
of the covariance matrix $\Sigma$ and studied their performance under various sparsity assumptions on $\Sigma$. See, for instance, \cite{Johnstone_Lu_arxiv,Lounici2013,Ma2013,Paul_Johnstone2012,VuLei2012} and the references cited therein. 

We turn now to formulating the PCA problem in a general separable Hilbert space ${\mathbb H}.$ 
This framework includes not only the classical high-dimensional setting, but also 
PCA for functional data (FPCA), see Ramsay and Silverman \cite{Ramsay}, and kernel PCA (KPCA) in machine learning, see Sch\"olkopf, Smola and M\"uller \cite{Scholkopf},
Blanchard, Bousquet and Zwald \cite{Blanchard}. 

It will be assumed that ${\mathbb H}$ is a real Hilbert space, but, in some cases 
(especially, when one has to deal with resolvents of operators in ${\mathbb H}$), 
it has to be extended to a complex Hilbert space ${\mathbb H}_{{\mathbb C}}:=\{u+i v: u,v\in {\mathbb H}\}$ with a standard extension of the inner product.
In what follows, $\langle \cdot,\cdot\rangle$ denotes the inner product of ${\mathbb H}$ with $\|\cdot\|$ being the corresponding norm. With a little
abuse of notation, we also denote by $\langle \cdot,\cdot\rangle$ the standard inner product in the space of Hilbert--Schmidt operators acting in ${\mathbb H},$
the corresponding Hilbert--Schmidt norm being denoted by $\|\cdot\|_2.$
The notation $\|\cdot\|_{\infty}$ will be used for the operator norm of 
linear operators:
$$
\|A\|_{\infty}:= \sup_{\|u\|\leq 1}\|Au\|,\ A:{\mathbb H}\mapsto {\mathbb H}.
$$ 
More generally, $\|\cdot\|_{p}, p\in [0,+\infty]$ denotes the Schatten 
$p$-norm. 
Given vectors $u,v\in {\mathbb H},$ $u\otimes v$
is the tensor product of $u$ and $v$ (that is, $u\otimes v$ is an operator from ${\mathbb H}$ into ${\mathbb H}$ acting as follows: $(u\otimes v)x= \langle v,x\rangle u, x\in {\mathbb H}$). If $P$ is the orthogonal projector on a 
subspace $L\subset {\mathbb H},$ then $P^{\perp}$ denotes the projector on 
the orthogonal complement $L^{\perp}.$

The following notations are used throughout the paper: for nonnegative $B_1,B_2,$ $B_1\lesssim B_2$ (equivalently, $B_2\gtrsim B_1$) means that there exists an absolute constant $C>0$ such that $B_1\leq CB_2.$ If $B_1\lesssim B_2$ and $B_1\gtrsim B_2,$ we will write $B_1\asymp B_2.$ Sometimes, the signs $\lesssim, \gtrsim$ and $\asymp$ will be provided with subscripts. For instance, $B_1\lesssim_a B_2$ would mean that $B_1\leq CB_2,$ where $C$ is a constant that might depend on $a.$

Let $X,X_1,\dots,X_n$
be i.i.d. random vectors in ${\mathbb H}$ with mean zero and ${\mathbb E}\|X\|^2<+\infty.$
Denote by $\Sigma={\mathbb E}(X\otimes X)$ the covariance matrix of $X$ and  
let 
$$
\hat \Sigma := \hat \Sigma_n:=n^{-1}\sum_{j=1}^n X_j\otimes X_j
$$
be the sample covariance based on the observations $(X_1,\dots, X_n).$
Since $\Sigma$ is a compact symmetric nonnegatively definite operator
(in fact, a trace class operator), 
it has the following spectral decomposition
$
\Sigma = \sum_{r=1}^{\infty}\mu_r P_r,
$
where $\mu_r=\mu_r(\Sigma)$ are distinct strictly positive eigenvalues 
of $\Sigma$ (to be specific, arranged in decreasing order) and $P_r$ are the corresponding spectral projectors (orthogonal projectors in ${\mathbb H}$). Clearly, $m_r:={\rm rank}(P_r)<+\infty$ is the multiplicity of the eigenvalue $\mu_r$ in the spectrum $\sigma(\Sigma)$ of 
$\Sigma$ (in other words, it is 
the dimension of the eigenspace of $\Sigma$ that corresponds to $\mu_r$).
It will be convenient in what follows to denote by $\sigma_j=\sigma_j(\Sigma), j\geq 1$
the eigenvalues of $\Sigma$ arranged in a nonincreasing order and \it repeated 
with their multiplicities. \rm Let $\Delta_r:=\{j:\sigma_j=\mu_r\}.$ Then ${\rm card}(\Delta_r)=m_r.$
Of course, the sample covariance $\hat \Sigma$ admits a similar spectral 
representation. Note that since the rank of $\hat \Sigma$ is at most $n,$ it has 
at most $n$ non-zero eigenvalues. Denote by $\hat P_r$ the orthogonal projector 
on the direct sum of eigenspaces of $\hat \Sigma$ corresponding to the eigenvalues 
$\{\sigma_j(\hat \Sigma), j\in \Delta_r\}.$ It is well known (and it will be discussed in 
detail in the next section) that as soon as $\hat \Sigma$ is close enough to $\Sigma$
in the operator norm, the eigenvalues $\{\sigma_j(\hat \Sigma), j\in \Delta_r\}$ are  
in a small neighborhood of $\mu_r$ and all other eigenvalues of $\hat \Sigma$ are 
separated from this neighborhood. Thus, for each $r,$ if $n$ is sufficiently large,
there is a cluster $\{\sigma_j(\hat \Sigma), j\in \Delta_r\}$ of eigenvalues of $\hat \Sigma$ 
and the corresponding spectral projector $\hat P_r$ is a natural estimator of $P_r$
(note that, in this case, ${\rm rank}(\hat P_r)={\rm rank}(P_r)=m_r$). 

We will be interested in asymptotic properties of the ``empirical'' spectral projector $\hat P_r$ as an estimator of the true spectral projector $P_r.$
The following assumption holds throughout the paper: 

\begin{assumption}
Assume that 
$X,X_1,\dots, X_n$
are i.i.d. random variables sampled from a Gaussian distribution in ${\mathbb H}$
with zero mean and covariance $\Sigma.$ 
\end{assumption}
 
We are especially interested in the case when not only the sample size $n$
is large, but also the trace of matrix $\Sigma,$ ${\rm tr}(\Sigma),$
is large as well (formally, one has to deal with a sequence of problems 
with covariances $\Sigma^{(n)}$ such that ${\rm tr}(\Sigma^{(n)})\to \infty$ as $n\to \infty$). This is a crucial difference with other literature on PCA in Hilbert spaces 
(such as \cite{DPR}) where it is typically assumed that ${\rm tr}(\Sigma)$ is a 
constant. This is what makes our results closer to what has been studied in the literature on PCA in high dimensions. To simplify
the matter, we will assume that the individual eigenvalues in the spectrum 
of $\Sigma$ are not large, so, the operator norm $\|\Sigma\|_{\infty}$ will be bounded by a constant. In this case, it makes sense to characterize 
the dimensionality of the problem by the so called {\it ``effective rank''}
of $\Sigma$ (which also tends to infinity).  

\begin{definition}
The following quantity
$$
{\bf r}(\Sigma):= \frac{{\rm tr}(\Sigma)}{\|\Sigma\|_{\infty}}
$$
will be called the effective rank of $\Sigma.$
\end{definition}

Clearly, ${\bf r}(\Sigma)\leq {\rm rank}(\Sigma).$  
Our setting includes, in particular, a popular high-dimensional {\it spiked covariance model} (see \cite{Johnstone}, \cite{Johnstone_Lu}, \cite{Paul_2007}) described in the following 
example. 

{\bf Example: Spiked Covariance Model.} 
Suppose that $\{\theta_k\}$ is an orthonormal basis in ${\mathbb H}$ and let 
$
S:= \sum_{k=1}^m s_k \zeta_k \theta_k
$
be a ``signal'', $s_j, j=1,\dots, m$ being nonrandom positive real numbers and $\zeta_j, j=1,\dots, m$
being i.i.d. standard normal random variables. Let $\dot W$ be a {\it Gaussian white noise} 
(a centered Gaussian r.v. with mean zero and identity covariance operator)
that could be informally written as $\dot W= \sum_{k\geq 1} \eta_k \theta_k,$ where $\{\eta_k\}$ are i.i.d. standard normal random variables (independent also of $\{\zeta_k\}$).
Note that $\dot W$ is not a random vector in ${\mathbb H},$ 
but the family of linear functionals $\langle \dot W, u\rangle, u\in {\mathbb H}$
is well defined as an isonormal Gaussian process indexed by ${\mathbb H},$
that is, a centered Gaussian process with covariance function 
$$
{\mathbb E}\langle \dot W, u\rangle \langle \dot W, v\rangle = \langle u,v \rangle,
u,v\in {\mathbb H}.
$$ 
Thus, $\dot W$ is defined in a ``weak sense'' and it is well known that it can be
also formally described as a random variable in a proper extension ${\mathbb H}_{-}\supset {\mathbb H}$ (often defined as a space of linear functionals on a dense linear subspace of ${\mathbb H}$). 
Suppose that $S$ is observed in additive ``white noise'', that is, the observation of $S$ is $X=S+\sigma \dot W.$ More precisely, we will assume that the data consists of i.i.d. copies 
$X_1^{(n)}, \dots, X_n^{(n)}$ of a random vector $X^{(n)}\in {\mathbb H},$ where
$$
X^{(n)}= S + \sigma {\dot W}^{(n)},\ \ {\dot W}^{(n)}=\sum_{k=1}^{p} \eta_k \theta_k,\ \ p=p_n\to\infty\ {\rm as}\ n\to\infty.
$$
It is easy to see that $X^{(n)}$ can be rewritten as 
$$
X^{(n)}= \sum_{j=1}^m \sqrt{s_j^2 + \sigma^2} \xi_j \theta_j + \sigma \sum_{j=m+1}^{p_n}\xi_j \theta_j, 
$$
where ${\xi_j}$ are i.i.d. standard normal random variables. The covariance of $X^{(n)}$ is 
$$
\Sigma^{(n)}={\mathbb E}(X^{(n)}\otimes X^{(n)})= \sum_{j=1}^m (s_j^2+\sigma^2) (\theta_j\otimes \theta_j)+ \sigma^2 P_{m,p_n},
$$ 
where $P_{m,p_n}$ denotes the orthogonal projector on the linear span of vectors $\theta_j, j=m+1,\dots, p_n.$ Clearly, for a fixed $m,$
$$
{\rm tr}(\Sigma^{(n)})= \sum_{j=1}^m s_j^2 + \sigma^2 p_n \asymp p_n \to \infty\ {\rm as}\ n\to \infty.
$$
Estimation of the vectors $\theta_1,\dots, \theta_m$ (the components of the ``signal'')
can be now viewed as a PCA problem for unknown covariance 
$\Sigma^{(n)}.$ Obviously, as it is usually done in the literature, one can also phrase this as a sequence of high-dimensional problems in spaces $\mathbb R^p, p=p_n$ 
(without an explicit embedding of ${\mathbb R}^p$ into an infinite dimensional Hilbert space ${\mathbb H}$).  
In such a high-dimensional setting, the performance of the PCA is usually assessed by measuring the ``alignment'' between the target eigenvector and its estimator. In \cite{Birnbaumetal}, the authors considered the loss function $L(a,b):= 2(1-|\langle a,b\rangle |),$ where $a,b \in \R^p$ are unit vectors. 
It is closely related to the loss function 
$$L'(a,b):= \|a \otimes a - b \otimes b\|_2^2 = 2(1-\langle a,b\rangle^2),$$ 
that is used, for instance, in \cite{MaCaiYihong,Lounici2013,VuLei2012}. For the spiked covariance model described above, where $s_1>\cdots>s_m>0$, $\sigma^2 = 1$ and $m\geq 1$ are fixed and $\frac{p}{n}\rightarrow 0$ as $n\rightarrow \infty,$ 
the following asymptotic representation of the risk of classical PCA was obtained in \cite{Birnbaumetal}:
\begin{equation}
\label{Birnbaumresult}
\E L(\hat\theta_j ,\theta_j) = \left[  \frac{(p-m)(1+s_j^2)}{n s_j^4}  + \frac{1}{n}\sum_{k\neq j} \frac{(1+s_j^2)(1+s_k^2)}{(s_j^2 -s_k^2)^2}\right](1+o(1)),\quad \forall 1\leq j \leq m.
\end{equation}
In \cite{Birnbaumetal}, the authors also considered the setting $\frac{p}{n}\rightarrow c>0$ as $n\rightarrow \infty,$ where the classical PCA is known to produce inconsistent estimators of the eigenvectors (see, for instance, \cite{Johnstone_Lu}), and proposed a thresholding procedure related to, but more refined than the diagonal thresholding of Johnstone and Lu \cite{Johnstone_Lu} that achieves optimality in the minimax sense for the loss $L(\cdot,\cdot)$ under sparsity conditions on the eigenvectors of $\Sigma$. 

The loss functions $L$ and $L'$ are not suitable for the {\em support recovery} problem, that is, the estimation of the set $
{\rm supp}(\theta_r):=\Bigl\{j: \theta_r^{(j)}\neq 0\Bigr\}
$
for an eigenvector $\theta_r.$ To the best of our knowledge, very few results on this problem   
are available in the high-dimensional setting and they are obtained under very restrictive conditions on the covariance structure. 
For instance, in \cite{AminiWainwright}, a spiked covariance model was considered, where $\Sigma = s_1^2 \theta_1 \otimes \theta_1 + \left(\begin{array}{c|c} I_k & 0\\\hline 0& \Gamma_{p-k}\end{array} \right)$, the first $k$ entries of $\theta_1 \in S^{p-1}$ are equal to $\pm \frac{1}{\sqrt{k}}$ for some $k\geq 1$ and $\Gamma_{p-k}$ is symmetric positive semi-definite with $\|\Gamma_{p-k}\|_{\infty} \leq 1$. The authors established an asymptotic support recovery result for the SDP-relaxation methodology introduced in \cite{aspremont}, assuming that $k= O(\log p)$ is known, that $n \geq C(\Sigma) k \log (p-k),$ where $C(\Sigma)>0$ depends only on $\Sigma,$ and also assuming 
the existence of a rank one solution of the SDP optimization problem. 

Asymptotics of eigenvectors of sample covariance in a high-dimensional spiked covariance model were studied by Paul \cite{Paul_2007}. Namely, he considered 
a problem, where $X\sim N_p(0,\Sigma)$ with a spiked covariance matrix 
$$\Sigma = \mathrm{diag}(s_1^2,s_2^2,\cdots,s_m^2,1,\cdots,1)$$ 
and fixed $s_1>\cdots>s_m>1$, $m\geq 1$. Let $\hat \theta_j$ be the $j$-th sample eigenvector and let $\hat\theta_j  = (\hat \theta_{A,j}, \hat\theta_{B,j}),$ where $\hat\theta_{A,j}$ is the subvector corresponding to the first $m$ components and $\hat\theta_{B,j}$ contains the remaining $p-m$ components. Paul \cite{Paul_2007} established that
$\frac{\hat\theta_{B,j}}{\|\hat\theta_{B,j}\|} $ is uniformly distributed in the unit sphere $S^{p-m-1}$ and is independent of $\|\hat\theta_{B,j}\|.$ In addition, if $\frac{p}{n} - c  = o(\frac{1}{n^{1/2}})$ with $c \in (0,1)$ and $s_j^2 > 1+ \sqrt{c}$, then also 
$$
\sqrt{n} \left( \frac{\hat\theta_{A,j}}{\|\hat\theta_{A,j}\|}  - e_j\right) \rightarrow N(0,\Sigma_j(s_j))\ {\rm as}\ n\rightarrow \infty,
$$
where 
$$
\Sigma_{j}(s_j) = \left( \frac{1}{1 - \frac{c}{(s_j^2 - 1)^2}}\right)\sum_{1\leq k\neq j \leq m} \frac{(s_k s_j)^2}{(s_k^2 - s_j^2)^2}(e_k\otimes e_k),
$$
and $e_k$ is the $k$-th vector of the canonical basis of $\R^p$.

The spiked covariance model is a special case of more general models discussed in the next example.

{\bf Example: More General Spiked Models.}
Let $\Sigma$ be a symmetric nonnegatively definite bounded operator 
that admits the following representation
$$
\Sigma = \sum_{r=1}^m \mu_r P_r + \Upsilon,
$$
where $\mu_r$ are distinct positive numbers, $P_r$ are projectors on 
mutually orthogonal finite dimensional subspaces of ${\mathbb H}$
and $\Upsilon:{\mathbb H}\mapsto {\mathbb H}$ is a nonnegatively definite 
symmetric bounded operator such that 
$P_r \Upsilon = \Upsilon P_r=0, r=1,\dots, m.$
Moreover, suppose that $\|\Upsilon\|_{\infty}< \min_{1\leq r\leq m}\mu_r$
(in which case the spectrum of $\Sigma$ is the union of two separated sets, 
$\{\mu_1,\dots, \mu_r\}$ and the spectrum of the operator $\Upsilon$). 
Note that since $\Upsilon$ is not necessarily of trace class, it might not be 
a covariance operator of a random vector in ${\mathbb H}$ with a bounded strong 
second moment, and the same applies 
to $\Sigma.$ However, $\Sigma$ and $\Upsilon$ can be always viewed as covariance operators 
of ``generalized random elements'' (linear functionals on dense linear subspaces of ${\mathbb H}$), the same way as the identity operator is the covariance operator of the white noise $\dot W.$ Let $P_{L_n}$ be the orthogonal projector on a finite-dimensional subspace $L_n\subset {\mathbb H}.$ 
Suppose that ${\rm dim}(L_n)\to \infty$ as $n\goin,$ $\bigcup_{n\geq 1}L_n$ 
is dense in ${\mathbb H}$ and 
$P_r {\mathbb H}\subset L_n, r=1,\dots, m$ for all large enough $n.$ Let $X^{(n)}$
be a centered Gaussian vector in ${\mathbb H}$ with covariance operator $\Sigma^{(n)}= P_{L_n}\Sigma P_{L_n}$ and let $X_1^{(n)}, \dots, X_n^{(n)}$ be i.i.d. copies of $X^{(n)}.$ 
Then the problem becomes to estimate the principal spectral projectors $P_r, r=1,\dots, m$
based on the sample $(X_1^{(n)}, \dots, X_n^{(n)}),$ which is again a PCA problem.
If ${\rm tr}(\Upsilon)=\infty,$ then also ${\rm tr}(\Sigma)=\infty$ and ${\rm tr}(\Sigma^{(n)})\to \infty$ as $n\goin .$ 
One can go even further and consider the case of more general covariance operators $\Sigma^{(n)}$ of the observations $X_1^{(n)}, \dots, X_n^{(n)}$ that converge in some sense (for instance, in the sense of strong convergence of operators) to a symmetric nonnegatively definite operator $\Sigma.$

In this paper and in a related paper \cite{Kol_Lou_2014}, we develop a general theory of the asymptotic behavior of spectral projectors of the sample covariance operators that encompasses the spike covariance models described above as well as more general models of covariance operators for observations in a separable Hilbert space. We are especially interested in the case 
when ${\bf r}(\Sigma^{(n)})=o(n),$ which is a necessary and sufficient condition 
for convergence of the sample covariance $\hat \Sigma_n$ to the true covariance $\Sigma$ in the operator norm (and which, essentially, implies consistency of eigenvalues and of spectral projectors of sample covariance as estimators of their population counterparts).  
More specifically, our contributions include the following:

\begin{itemize}

\item 
In Section \ref{Sec:Prelim}, we review recent moment bounds and concentration 
inequalities (see \cite{Koltchinskii_Lounici_arxiv}) for $\|\hat\Sigma_n-\Sigma\|_{\infty}$ 
showing that, in the Gaussian case, the size of this random variable is completely characterized by two parameters, the operator norm $\|\Sigma\|_{\infty}$ and the effective rank ${\bf r}(\Sigma).$ This implies that $\|\hat \Sigma_n-\Sigma\|_{\infty}\to 0$ (a.s. and in the mean) if and only if ${\bf r}(\Sigma)=o(n).$ In the same section, we discuss several results in perturbation theory used throughout the paper.

\item 
In Section \ref{Sec:Representation}, we obtain basic concentration inequalities 
for bilinear forms of empirical spectral projectors $\hat P_r.$ In particular, we show that the following representation holds:
$$
\hat P_r -{\mathbb E}P_r = L_r + R_r,
$$
where the main term $L_r$ is linear with respect to $\hat \Sigma-\Sigma$ and, thus,
it can be represented as a sum of i.i.d. random variables. The bilinear forms of the remainder term $R_r$ satisfy sharp Gaussian type concentration inequalities, implying, in particular, that 
$$
\Bigl|\langle R_r u,v\rangle\Bigr|=
O_{\mathbb P}\biggl(\sqrt{\frac{{\bf r}(\Sigma)}{n}}\sqrt{\frac{1}{n}}\biggr). 
$$ 
If ${\bf r}(\Sigma)=o(n),$ the bilinear forms $\langle R_r u,v\rangle$ are 
of the order $o_{\mathbb P}(n^{-1/2})$ and asymptotic normality of the bilinear forms $\Bigl\langle(\hat P_r -{\mathbb E}P_r) u,v\Bigr\rangle$ can be easily deduced from the central limit 
theorem applied to the linear term $\langle L_r u,v\rangle.$

\item
In Section \ref{Sec:Bias}, we derive an asymptotic representation for the bias 
${\mathbb E}\hat P_r-P_r$ of the empirical spectral projector $\hat P_r$ showing 
that its main term is an operator of the form $P_r W_r P_r,$ where $\|W_r\|_{\infty}=O(\frac{{\bf r}(\Sigma)}{n}),$ and the remainder is 
of the order $O\biggl(\sqrt{\frac{{\bf r}(\Sigma)}{n}}\sqrt{\frac{1}{n}}\biggr).$  
This implies, in particular, that, in the case when $m_r=1$ (the case of simple eigenvalue) the bias is proportional to the one-dimensional true spectral projector 
$P_r$ up to a higher order term (indicating that a multiplicative correction can lead to a bias reduction).

\item
In Sections \ref{Sec:Asymptotic} 
we derive the asymptotic distributions of bilinear forms of the empirical spectral projectors. In particular, we show that, under the assumption ${\bf r}(\Sigma)=o(n),$ the finite dimensional distributions of 
$$\sqrt{n}\Bigl\langle (\hat P_r - {\mathbb E}\hat P_r)u,v \Bigr\rangle, u,v\in {\mathbb H}$$ converge weakly to the finite dimensional distributions of a Gaussian process. 
Our results show that the ``variance part'' of the error 
$\Bigl\langle (\hat P_r -P_r)u,v \Bigr\rangle$ 
is relatively well-behaved and that its dominating part is ``bias'', which might require further attention in statistical applications.


\item 
In Section \ref{Sec:VarSel}, we study in more detail the case of spectral projectors corresponding to an isolated eigenvalue of multiplicity $m_r=1.$ In this case, we prove the asymptotic normality of properly centered and normalized linear forms 
$\langle \hat \theta_r,u\rangle, u\in {\mathbb H}$ of the corresponding sample eigenvector $\hat \theta_r.$ Namely, we prove the weak convergence of finite dimensional 
distributions of stochastic processes 
$$
n^{1/2}\Bigl\langle \hat \theta_r-\sqrt{1+b_r}\theta_r,u\Bigr\rangle, u\in {\mathbb H} $$
to the finite dimensional distributions of a Gaussian process for properly chosen 
``bias parameters'' $b_r.$
We also obtain non-asymptotic concentration bounds for the $l_{\infty}$-norm $\Bigl\|\hat \theta_r-\sqrt{1+b_r}\theta_r\Bigr\|_{\ell_{\infty}}.$ 
In addition, we propose an estimator of the bias parameter $b_{r}$
that converges to the true parameter at a rate faster than $n^{-1/2}$ and develop a bias reduction method based on this estimator.
At the end of Section \ref{Sec:VarSel}, we briefly discuss potential applications of these results, in particular, to the problem of support recovery of the eigenvector of interest as well as sparse PCA estimation.

\end{itemize}

In a related paper \cite{Kol_Lou_2014}, we 
obtained an asymptotic formula for the Hilbert--Schmidt norm risk ${\mathbb E}\|\hat P_r-P_r\|_2^2$ of empirical spectral projectors under the assumption that ${\bf r}(\Sigma)=o(n).$ In a special case of 
spiked covariance model, it implies representation (\ref{Birnbaumresult}). 
We also proved in \cite{Kol_Lou_2014} the asymptotic normality of a properly normalized sequence 
$$\Bigl\{\|\hat P_r - P_r\|_2^2-{\mathbb E}\|\hat P_r - P_r\|_2^2\Bigr\}.$$

\section{Preliminaries}\label{Sec:Prelim}

In this section, we review bounds on the operator norm $\|\hat \Sigma_n-\Sigma\|_{\infty}$ and 
discuss several well known facts of perturbation theory that will be frequently 
used in what follows.

\subsection{Bounds on the operator norm $\|\hat \Sigma_n-\Sigma\|_{\infty}.$}\label{Sec:Spectral}

It is well known (see \cite{vershynin}) that, for a sub-subgaussian isotropic distribution 
(that is, in the case when $\Sigma =I_p$), with probability at least $1-e^{-t}$ 
\begin{equation}
\label{vershynin}
\|\hat \Sigma_n - \Sigma\|_{\infty } \leq C \left(\sqrt{\frac{p}{n}}\bigvee \frac{p}{n}\bigvee \sqrt{\frac{t}{n}}\bigvee \frac{t}{n} \right),
\end{equation}
for some numerical constant $C>0$ (see Theorem 5.39 and the comments after this theorem). 
The proof is based on an $\varepsilon$-net argument that does not yield an optimal bound for general (nonisotropic) sub-subgaussian distributions. 
In \cite{Lounici2012, BuneaXiao}, similar results were derived for sub-subgaussian distributions and low-rank covariance matrices. However the bounds in the last two papers are suboptimal by a logarithmic factor (they are based on a noncommutative Bernstein inequality).

The following theorems (see Koltchinskii and Lounici \cite{Koltchinskii_Lounici_arxiv})
could be viewed as an extension of bound (\ref{vershynin}) to the nonisotropic and infinite-dimensional case. These results show that in the Gaussian case, the size of the operator norm $\|\hat \Sigma_n-\Sigma\|_{\infty}$ is completely characterized 
by the operator norm $\|\Sigma\|_{\infty}$ and the effective rank ${\bf r}(\Sigma).$
In particular, if $\Sigma=\Sigma^{(n)}$ with $\|\Sigma^{(n)}\|_{\infty}$ uniformly 
bounded, then $\|\hat \Sigma_n-\Sigma^{(n)}\|_{\infty}\to 0$ a.s. as $n\to\infty$ if and only if ${\bf r}(\Sigma^{(n)})=o(n).$

\begin{theorem}
\label{th_operator}
Let $X,X_1,\ldots,X_n$ be i.i.d. centered Gaussian random vectors in ${\mathbb H}$ with covariance $\Sigma = \E(X\otimes X).$ Then, for all $p\geq 1,$ 
\begin{align}
\label{E1/p}
\E^{1/p}\|\hat\Sigma_n - \Sigma\|_{\infty}^p \asymp_{p} \|\Sigma\|_\infty\max\left\lbrace \sqrt{\frac{\mathbf{r}(\Sigma)}{n}}, \frac{\mathbf{r}(\Sigma)}{n}\right\rbrace.
\end{align}
\end{theorem}

We will also need a concentration inequality for  
$\|\hat\Sigma_n - \Sigma\|_{\infty}.$

\begin{theorem}
\label{spectrum_sharper} 
Let $X,X_1,\ldots,X_n$ be i.i.d. centered Gaussian random vectors in ${\mathbb H}$ with  covariance $\Sigma = \E(X\otimes X).$  Then, there exist a constant 
$C>0$ such that for all $t\geq 1$ with probability at least $1-e^{-t},$ 
\begin{align}
 \label{con_con}
\Bigl|\|\hat\Sigma_n - \Sigma\|_{\infty}- {\mathbb E}\|\hat\Sigma_n - \Sigma\|_{\infty}\Bigr| \leq C\|\Sigma\|_\infty
\biggl[\biggl(\sqrt{\frac{\mathbf{r}(\Sigma)}{n}}\bigvee 1\biggr)\sqrt{\frac{t}{n}}\bigvee \frac{t}{n} 
\biggr].
\end{align}
As a consequence of this bound and (\ref{E1/p}), with some constant $C>0$ and with the same probability 
\begin{align}
\label{sha_sha}
\|\hat\Sigma_n - \Sigma\|_{\infty} \leq C\|\Sigma\|_\infty
\biggl[\sqrt{\frac{\mathbf{r}(\Sigma)}{n}} \bigvee \frac{\mathbf{r}(\Sigma)}{n} 
\bigvee \sqrt{\frac{t}{n}}\bigvee \frac{t}{n} 
\biggr].
\end{align}

\end{theorem}

\begin{remark}
\label{assumption1}

1. The notion of effective rank ${\bf r}(\Sigma)$ and the results of 
theorems \ref{th_operator} and \ref{spectrum_sharper} can be extended 
to the case of Gaussian random variables in separable Banach spaces,
see \cite{Koltchinskii_Lounici_arxiv}.

2. The bound of Theorem \ref{th_operator} and bound (\ref{sha_sha}) of Theorem \ref{spectrum_sharper} hold in a more general case, when $X,X_1,\dots, X_n$ are i.i.d. centered subgaussian 
vectors in ${\mathbb H},$ that is, for some constant $c>0,$
\begin{equation}\label{subexp1}
\|\langle X, u\rangle\|_{\psi_2}^2\leq c\E\langle X, u \rangle^2, 
u\in {\mathbb H}.
\end{equation}
Here $\|\cdot\|_{\psi_2}$ is the Orlicz norm for $\psi_2(t)=e^{t^2}-1, t\geq 0$
(the Orlicz norm in the space of subgaussian random variables).  
\end{remark}

\subsection{Several facts on perturbation theory}\label{sec:pert}

In this section, we discuss several useful results of perturbation theory 
(see Kato \cite{Kato}) adapted for our purposes. Some facts in the same direction can be found in Koltchinskii \cite{Koltchinskii} and Kneip and Utikal 
\cite{Kneip}.

Let $\Sigma :{\mathbb H}\mapsto {\mathbb H}$ be a compact symmetric operator 
(in applications, it will be the covariance operator of a random vector $X$ in ${\mathbb H}$).
Let $\sigma(\Sigma)$ be the spectrum of $\Sigma.$ It is well known that the following spectral representation holds 
$$
\Sigma = \sum_{r\geq 1}\mu_r P_r
$$
with distinct non-zero eigenvalues $\mu_r$ and spectral projectors 
$P_r$ and with the series converging in the operator norm. 
We will also use notations $\sigma_i=\sigma_i(\Sigma),$ $\Delta_r,$
$m_r,$ etc., already introduced in Section \ref{Sec:Intro}.

Define 
$$
g_r: = g_r(\Sigma) := \mu_r-\mu_{r+1}>0, r\geq 1. 
$$
Let $\bar g_r := \bar g_r(\Sigma):= \min(g_{r-1},g_r)$ for $r\geq 2$ and 
$\bar g_1:=g_1.$ In what follows, $\bar g_r$ will be called {\it the $r$-th  spectral gap, or the spectral gap of eigenvalue $\mu_r$}. 

Let now $\tilde \Sigma$ be another compact symmetric operator in ${\mathbb H}$
with spectrum $\sigma (\tilde \Sigma)$ and eigenvalues $\tilde \sigma_i=\sigma_i(\tilde \Sigma), i\geq 1$
(arranged in nonincreasing order and repeated with their multiplicities). Denote 
$E:=\tilde \Sigma-\Sigma.$ According to well known Lidskii's inequality,
$$
\sup_{j\geq 1}|\sigma_j(\Sigma)-\sigma_j(\tilde \Sigma)|\leq \sup_{j\geq 1}|\sigma_j(E)|
=\|E\|_{\infty}.
$$
This implies that, for all $r\geq 1,$
\begin{align*}
\inf_{j\not\in \Delta_r}|\tilde \sigma_j - \mu_r| \geq 
\bar g_r - \sup_{j\geq 1}|\tilde \sigma_j - \sigma_j| \geq 
\bar g_r -\|E\|_{\infty} 
\end{align*}
and 
\begin{align*}
\sup_{j\in \Delta_r}|\tilde \sigma_j - \mu_r| = \sup_{j\in \Delta_r}|\tilde \sigma_j-\sigma_j|\leq \|E\|_{\infty}. 
\end{align*}
Suppose that 
\begin{equation}
\label{bound_on_E}
\|E\|_{\infty}< \frac{\bar g_r}{2}.
\end{equation}
Then, all the eigenvalues $\tilde \sigma_j, j\in \Delta_r$ are covered by 
an interval 
$$\Bigl(\mu_r-\|E\|_{\infty},\mu_r+\|E\|_{\infty}\Bigr)\subset (\mu_r-\bar g_r/2,\mu_r+\bar g_r/2)$$ 
and the rest of the eigenvalues of 
$\tilde \Sigma$ are outside of the interval 
$$\Bigl(\mu_r-(\bar g_r-\|E\|_{\infty}),\mu_r+(\bar g_r-\|E\|_{\infty})\Bigr)
\supset [\mu_r-\bar g_r/2,\mu_r+\bar g_r/2].$$ 
Moreover, if 
$$
\|E\|_{\infty}<\frac{1}{4}\min_{1\leq s\leq r}\bar g_s=:\bar \delta_r,
$$
then the set $\{\sigma_j(\tilde \Sigma): j\in \bigcup_{s=1}^r\Delta_s\}$ of  the largest eigenvalues of $\tilde \Sigma$ will be divided into $r$ clusters,
each of them being of diameter strictly smaller than $2\bar \delta_r$ and the distance between any two clusters being larger than $2\bar \delta_r.$
In principle, this allows one to identify clusters of eigenvalues of $\tilde \Sigma$ corresponding to each of the $r$ largest distinct eigenvalues $\mu_s, s=1,\dots, r$
of $\Sigma.$ 

Denote $\tilde P_r$ the orthogonal projector
on the direct sum of eigenspaces of $\tilde \Sigma$ corresponding to the eigenvalues 
$\tilde \sigma_j, j\in \Delta_r$ 
(in other words, to the $r$-th cluster of eigenvalues of $\Sigma$). 
Denote also 
$$
C_r:=\sum_{s\neq r}\frac{1}{\mu_r-\mu_s}P_s.
$$

\begin{lemma}\label{lem-pert-spectral}
The following bound holds: 
\begin{align}
\label{bd_1}
\|\tilde P_r-P_r\|_\infty \leq 4\frac{\|E\|_{\infty}}{\bar g_r}.
\end{align}
Moreover, 
\begin{equation}
\label{perture}
\tilde P_r-P_r =L_r(E)+ S_r(E),
\end{equation}
where 
\begin{align}
\label{linear_perturb}
L_r(E):=C_r E P_r + P_r E C_r
\end{align}
and 
\begin{equation}
\label{remainder_A}
\|S_r(E)\|_{\infty}\leq 14 \biggl(\frac{\|E\|_{\infty}}{\bar g_r}\biggr)^2.
\end{equation}
\end{lemma}

\begin{proof}
Assume first that $\|E\|_{\infty}\leq \bar g_r/4.$
Denote by $\gamma_r$ the circle in ${\mathbb C}$
with center $\mu_r$ and radius $\frac{\bar g_r}{2}.$ 
Note that the eigenvalues $\mu_r$ of $\Sigma$ and $\tilde \sigma_j, j\in \Delta_r$
of $\tilde \Sigma$ are inside this circle while the rest of the eigenvalues of these
operators are outside. 
Combining these facts with the Riesz formula for spectral projectors (see, for instance, \cite{Kato}, p.39), we get that
\begin{align*}
\tilde P_r = -\frac{1}{2\pi i} \oint_{\gamma_r} R_{\tilde\Sigma}(\eta)d\eta.
\end{align*}
where $R_A(\eta) = (A-\eta I)^{-1}$ is the resolvent of an operator $A$ in ${\mathbb H}.$  

The following computation is standard:
\begin{align}
\label{resolve}
R_{\tilde \Sigma}(\eta) &= R_{\Sigma+ E}(\eta) = (\Sigma+ E-\eta I)^{-1}\\
\nonumber
&= \left[ (\Sigma - \eta I)(I +(\Sigma - \eta I)^{-1}E)  \right]^{-1}\\
\nonumber
&= (I +R_{\Sigma}(\eta)E)^{-1}R_{\Sigma}(\eta)\\
\nonumber
&= \sum_{k\geq 0} (-1)^{k} [R_{\Sigma}(\eta)E]^k R_{\Sigma}(\eta),\quad \eta\in\gamma_r.
\end{align}
The series in the right hand side converges 
absolutely in the operator norm since 
$$
\| R_{\Sigma}(\eta)E \|_\infty \leq \|R_{\Sigma}(\eta)\|_\infty \|E\|_\infty \leq \frac{2}{\bar g_r}\|E\|_\infty \leq \frac{1}{2} <1,\quad \eta\in\gamma_r,
$$
where we have used that $\|R_{\Sigma}(\eta)\|_{\infty} \leq \frac{2}{\bar g_r}$ for any $\eta \in \gamma_r.$ 
Next, we get from (\ref{resolve}) that
\begin{align*}
\tilde P_r &= -\frac{1}{2\pi i} \oint_{\gamma_r} R_{\Sigma}(\eta)d\eta 
- \frac{1}{2\pi i} \oint_{\gamma_r}\sum_{k\geq 1} (-1)^{k} [R_{\Sigma}(\eta)E]^k R_{\Sigma}(\eta)d\eta\\
&=
P_r - \frac{1}{2\pi i} \oint_{\gamma_r}\sum_{k\geq 1} (-1)^{k} [R_{\Sigma}(\eta)E]^k R_{\Sigma}(\eta)d\eta,
\end{align*}
where we again used the Riesz formula. 
Thus, 
\begin{align*}
\|\tilde P_r - P_r\|_{\infty} &\leq 
2\pi\frac{\bar g_r}{2}\frac{1}{2\pi} \left(\frac{2}{\bar g_r}\right)^2\|E\|_{\infty}\sum_{k=0}^{\infty}\left(\frac{2}{\bar g_r}\|E\|_{\infty}\right)^k\\
&\leq \frac{2\|E\|_{\infty}/\bar g_r}{1- 2\|E\|_{\infty}/\bar g_r}.
\end{align*}
Under the assumption $\|E\|_{\infty}\leq \bar g_r/4,$ we get that 
\begin{equation}
\nonumber
\|\tilde P_r - P_r\|_{\infty} \leq 
\frac{4\|E\|_{\infty}}{\bar g_r},
\end{equation}
so, (\ref{bd_1}) holds in this case. 
Since $\tilde P_r, P_r$ are both orthogonal projectors, it is  
easy to see that $\|\tilde P_r - P_r\|_{\infty} \leq 1,$ implying that 
(\ref{bd_1}) also holds when $\|E\|_{\infty}> \bar g_r/4.$

We turn to the proof of the remaining bounds. 
It is easy to check (using the orthogonality of operators $C_r E P_r, P_r E C_r$)
that 
$$
\|L_r(E)\|_{\infty}=\|C_r E P_r + P_r E C_r\|_{\infty}\leq 
\sqrt{2}\|C_r\|_{\infty}\|E\|_{\infty} \leq \frac{\sqrt{2}}{\bar g_r}\|E\|_{\infty}.
$$
Therefore,
\begin{equation}
\label{trivia}
\|S_r(E)\|_{\infty}=\|\tilde P_r-P_r-L_r(E)\|_{\infty}
\leq \|\tilde P_r-P_r\|_{\infty} + \|L_r(E)\|_{\infty}
\leq 1+ \frac{\sqrt{2}}{\bar g_r}\|E\|_{\infty}.
\end{equation}

Assuming that $\|E\|_{\infty}\leq g_r/3,$
we have the following representation:
\begin{equation}
\hat P_r-P_r =L_r'(E)+ S_r'(E),
\end{equation}
where 
$$
L_r'(E)=
\frac{1}{2\pi i} \oint_{\gamma_r}R_{\Sigma}(\eta)ER_{\Sigma}(\eta)d\eta
$$
and 
$$
S_r'(E):=
- \frac{1}{2\pi i} \oint_{\gamma_r}\sum_{k\geq 2} (-1)^{k} [R_{\Sigma}(\eta)E]^k R_{\Sigma}(\eta)d\eta.
$$
As for the first order linear term $L_r'(E),$ we use the spectral representation 
of the resolvent $R_{\Sigma}(\eta),$
$$
R_{\Sigma}(\eta)=\sum_{j\geq 1} \frac{1}{\mu_j-\eta}P_j
$$ 
(with the series convergent in operator norm uniformly in $\eta\in \gamma_r$),
to derive that 
$$
L_r'(E)= \frac{1}{2\pi i} \oint_{\gamma_r}\sum_{j\geq 1}\frac{1}{\mu_j-\eta}P_j E
\sum_{j\geq 1}\frac{1}{\mu_j-\eta}P_j
d\eta
$$
$$
=\sum_{j_1,j_2\geq 1}
\frac{1}{2\pi i} \oint_{\gamma_r}\frac{d\eta}{(\mu_{j_1}-\eta)(\mu_{j_2}-\eta)}P_{j_1} E P_{j_2}.
$$
Note that, if $j_1=r, j_2=s\neq r,$ then, by Cauchy formula,
$$
\frac{1}{2\pi i}\oint_{\gamma_r}\frac{d\eta}{(\mu_{j_1}-\eta)(\mu_{j_2}-\eta)}=
\frac{1}{2\pi i}\oint_{\gamma_r}\frac{d\eta}{(\eta-\mu_r)(\eta-\mu_s)}=\frac{1}{\mu_r-\mu_s}.
$$
Similarly, if $j_2=r, j_1=s\neq r,$ then
$$
\frac{1}{2\pi i}\oint_{\gamma_r}\frac{d\eta}{(\mu_{j_1}-\eta)(\mu_{j_2}-\eta)}=
\frac{1}{\mu_r-\mu_s}.
$$
In all other cases, 
$$
\frac{1}{2\pi i}\oint_{\gamma_r}\frac{d\eta}{(\mu_{j_1}-\eta)(\mu_{j_2}-\eta)}=0.
$$
Therefore,
$$
L_r'(E)=\sum_{s\neq r} \frac{1}{\mu_r-\mu_s} P_sEP_r +\sum_{s\neq r} \frac{1}{\mu_r-\mu_s} P_rEP_s = C_r E P_r + P_r E C_r=L_r(E)
$$
and, as a consequence, $S_r'(E)=S_r(E).$
Similarly to (\ref{bd_1}), it can be proved that, under the assumption
$\|E\|_{\infty}\leq \bar g_r/3,$ 
\begin{equation}
\label{remainder_A'}
\|S_r(E)\|_{\infty}\leq 12\biggl(\frac{\|E\|_{\infty}}{\bar g_r}\biggr)^2.
\end{equation}
Bound (\ref{remainder_A}) now easily follows from (\ref{remainder_A'}) and 
(\ref{trivia}). 

\qed
\end{proof}

We will state below a simple generalization of Lemma \ref{lem-pert-spectral}. 
Given $I=\{r_1,r_1+1,\dots, r_2\}\subset {\mathbb N},$ $1\leq r_1\leq r_2,$
denote $\Delta_I:=\{j: \sigma_j=\mu_r, r\in I\}$ and let $P_I=\sum_{r\in I}P_r$
be the orthogonal projector on the direct sum of the eigenspaces of $\Sigma$ corresponding to the eigenvalues $\mu_r, r\in I.$ Denote $L_I:=\mu_{r_1}-\mu_{r_2}$
and define
$$
\bar g_I := \min\Bigl(\mu_{r_2}-\mu_{r_2+1},\mu_{r_1-1}-\mu_{r_1}\Bigr)\ {\rm if}\ r_1>1\ {\rm and}\ \bar g_I:=\mu_{r_2}-\mu_{r_2+1}\ {\rm if}\ r_1=1.
$$
Finally, let $\tilde P_I$ be the orthogonal projector on the direct sum of the eigenspaces of $\tilde \Sigma$ corresponding to the eigenvalues $\tilde \sigma_j, j\in \Delta_I.$
Note that, if $\|E\|_{\infty}<\bar g_I/2,$ then the set of eigenvalues $\{\tilde \sigma_j:j\in \Delta_I\}$ is covered by the interval 
$(\mu_{r_2}-\bar g_I/2, \mu_{r_1}+\bar g_I/2)$ and the rest of the eigenvalues of $\tilde \Sigma$ are outside of the interval $[\mu_{r_2}-\bar g_I/2, \mu_{r_1}+\bar g_I/2].$ Denote 
$$
\gamma_I:=\Bigl\{\eta\in {\mathbb C}: {\rm dist}(\eta;[\mu_{r_2},\mu_{r_1}])=\bar g_I/2\Bigr\}.
$$
In what follows, $\gamma_I$ will be viewed as a counter-clockwise contour and in (\ref{linear_perturb'}) below it can be replaced by an arbitrary contour $\gamma$ that separates the eigenvalues $\{\mu_r: r\in I\}$ from the rest of the spectrum of 
$\Sigma.$

\begin{lemma}\label{lem-pert-spectral-general}
The following bound holds: 
\begin{align}
\label{bd_1'}
\|\tilde P_I-P_I\|_\infty \leq 4\biggl(1+\frac{2}{\pi}\frac{L_I}{\bar g_I}\biggr)\frac{\|E\|_{\infty}}{\bar g_I}.
\end{align}
Moreover, the following representation holds
\begin{equation}
\label{perture'}
\tilde P_I-P_I =L_I(E)+ S_I(E),
\end{equation}
where the linear part $L_I(E)$ is given by 
\begin{align}
\label{linear_perturb'}
L_I(E):=\frac{1}{2\pi i} \oint_{\gamma_I}R_{\Sigma}(\eta)ER_{\Sigma}(\eta)d\eta
\end{align}
and the remainder $S_I(E)$ satisfies the bound
\begin{equation}
\label{remainder_A''}
\|S_I(E)\|_{\infty}\leq 15 \biggl(1+\frac{2}{\pi}\frac{L_I}{\bar g_I}\biggr)\biggl(\frac{\|E\|_{\infty}}{\bar g_I}\biggr)^2.
\end{equation}
\end{lemma}

The proof of this lemma is quite similar to the proof of Lemma \ref{lem-pert-spectral} and it will be skipped.

\section{Concentration Inequalities for Bilinear Forms of Empirical Spectral Projectors}\label{Sec:Representation}

Let $\hat P_r$ be the orthogonal projector on the direct sum 
of eigenspaces of $\hat \Sigma$ corresponding to the eigenvalues 
$\{\sigma_j(\hat \Sigma), j\in \Delta_r\}$ (in other words, to the $r$-th 
cluster of eigenvalues of $\hat \Sigma,$ see Section \ref{sec:pert}).

The goal of this section is to derive useful representations and concentration bounds for the bilinear forms
$\Bigl\langle(\hat P_r - P_r)u,v\Bigr\rangle, u,v\in {\mathbb H}$
of spectral projectors for a properly isolated eigenvalue $\mu_r.$
These results will 
be used in subsequent sections to show asymptotic normality 
of the bilinear forms $\Bigl\langle(\hat P_r - P_r)u,v\Bigr\rangle$
under the assumption that ${\bf r}(\Sigma)=o(n).$

In the results below, it will be assumed that, for some $\gamma \in (0,1)$,
\begin{align}\label{condition gamma_r-1}
\E \|\hat \Sigma - \Sigma\|_{\infty} \leq \frac{(1-\gamma)\bar g_r}{2}.
\end{align}
In view of Theorem \ref{th_operator}, this assumption implies that  
$$
\|\Sigma\|_{\infty}\biggl(\sqrt{\frac{{\bf r}(\Sigma)}{n}}\bigvee \frac{{\bf r}(\Sigma)}{n}\biggr)
\lesssim \frac{\bar g_r}{2}\leq \|\Sigma\|_{\infty}.
$$
Hence, we have ${\bf r}(\Sigma)\lesssim n$. Theorem \ref{spectrum_sharper} implies that for some constant $C'>0$ and for all $t\geq 1$ with probability at least $1-e^{-t}$
$$
\| \hat \Sigma -\Sigma\|_{\infty} \leq \E \| \hat \Sigma -\Sigma\|_{\infty}  + C' \|\Sigma\|_{
\infty}\left( \sqrt{\frac{t}{n}}\vee \frac{t}{n}  \right).
$$
If $$C' \|\Sigma\|_{\infty}\left(\sqrt{\frac{t}{n}} \vee \frac{t}{n} \right) \leq \frac{\gamma \bar g_r}{2},$$ then  $\mathbb{P}\left(\|\hat \Sigma_n-\Sigma\|_{\infty}<\frac{\bar g_r}{2}\right)\geq 1-e^{-t}$. It was pointed out in Section \ref{sec:pert} that, in this case,
the cluster $\{\sigma_j(\hat \Sigma_n):j\in \Delta_r\}$ of eigenvalues of 
$\hat \Sigma$ is well separated from the rest of the spectrum of $\hat \Sigma$
and the spectral projector $\hat P_r$ can be viewed as an estimator of the spectral projector $P_r$ (in particular, these two projectors are of the same rank $m_r$). 
It will be shown below that, under such an assumption, 
the bilinear form $\Bigl\langle(\hat P_r - P_r)u,v\Bigr\rangle$ can be represented as 
a sum of a part that is linear in $\hat \Sigma_n-\Sigma$ and a remainder that 
is smaller than the linear part, provided that ${\bf r}(\Sigma)=o(n).$  
The linear part is defined in terms of operator
$$
L_r:= C_r(\hat \Sigma-\Sigma)P_r+P_r(\hat \Sigma-\Sigma)C_r= n^{-1}\sum_{j=1}^n (C_r X_j\otimes P_r X_j+ P_r X_j \otimes C_r X_j)
$$
and the remainder in terms of operator
$$
R_r := (\hat P_r - P_r)-{\mathbb E}(\hat P_r-P_r)-L_r= \hat P_r-{\mathbb E}\hat P_r-L_r. 
$$

\begin{theorem}
\label{technical_1}
Suppose that, for some $\gamma\in (0,1),$ (\ref{condition gamma_r-1}) is satisfied.
Then, there exists a constant $D_{\gamma}>0$ such that, for all $u,v\in {\mathbb H},$ 
the following bound holds with probability at least $1-e^{-t}:$ 
\begin{equation}
\label{remaind}
|\langle R_r u,v\rangle |\leq 
D_\gamma \frac{\|\Sigma\|_{\infty}^2}{\bar g_r^2}
\biggl(\sqrt{\frac{{\bf r}(\Sigma)}{n}}\bigvee \sqrt{\frac{t}{n}}\bigvee \frac{t}{n}\biggr)
\sqrt{\frac{t}{n}}\|u\|\|v\|.
\end{equation}
\end{theorem}

Taking into account Theorem \ref{spectrum_sharper}, note that,
if $\Sigma=\Sigma^{(n)},$
$\|\Sigma^{(n)}\|_{\infty}=O(1),$ $\bar g_r=\bar g_r^{(n)}$ is bounded away from zero and 
${\bf r}(\Sigma^{(n)})\leq cn$ for a sufficiently small $c,$ then bound (\ref{remaind}) implies that
$$
\langle R_r u,v\rangle =O_{\mathbb P}(n^{-1/2})\ {\rm as}\ n\to\infty, u,v\in {\mathbb H}.
$$
Moreover, if ${\bf r}(\Sigma^{(n)})=o(n),$ 
it follows from (\ref{remaind})
that 
\begin{equation}
\label{Rr}
\langle R_r u,v\rangle =o_{\mathbb P}(n^{-1/2}).
\end{equation}

Let 
$$
\xi(u,v):= 
\langle X,P_r v\rangle \langle X, C_r u\rangle, u,v\in {\mathbb H}
$$ 
and let 
$$
\xi_j(u,v):=\langle X_j,P_r v\rangle \langle X_j, C_r u\rangle, u,v\in {\mathbb H}, j=1,\dots, n
$$
be independent copies of $\xi.$
Note that 
$$
{\mathbb E}\xi(u,v)={\mathbb E}\langle X,P_r v\rangle 
{\mathbb E}\langle X, C_ru\rangle
=0
$$
and 
$$
{\mathbb E}\xi(u,v)\xi(u',v')={\mathbb E}\langle X,P_r v\rangle \langle X,P_r v'\rangle
{\mathbb E}\langle X, C_ru\rangle\langle X, C_ru'\rangle 
$$
$$
=\langle P_r\Sigma P_r v,v'\rangle \langle C_r\Sigma C_ru,u'\rangle,
$$
where it was used that Gaussian random variables $\langle X,P_r v\rangle, \langle X, C_r u\rangle$ are uncorrelated and, hence, independent. 
This implies that the covariance function of the random field $\xi(u,v)+\xi(v,u), u,v\in {\mathbb H}$ is given by
\begin{align}
\nonumber
&\tilde \Gamma (u,v;u',v'):=
{\mathbb E}(\xi(u,v)+\xi(v,u))(\xi(u',v')+\xi(v',u'))
=
\\
\nonumber
&
=
\langle P_r\Sigma P_r v,v'\rangle \langle C_r\Sigma C_ru,u'\rangle
+ 
\langle P_r\Sigma P_r v,u'\rangle \langle C_r\Sigma C_ru,v'\rangle
\\
\nonumber
&
+
\langle P_r\Sigma P_r u,u'\rangle \langle C_r\Sigma C_rv,v'\rangle
+
\langle P_r\Sigma P_r u,v'\rangle \langle C_r\Sigma C_rv,u'\rangle.
\end{align}
The bilinear forms 
$$
n^{1/2}\Bigl\langle L_r u,v\Bigr\rangle =n^{-1/2}\sum_{j=1}^n (\xi_j(u,v)+\xi_j(v,u)),\ u,v\in {\mathbb H}
$$
have the same covariance function $\tilde \Gamma.$  
Moreover, it is easy to see that, under proper assumptions, 
they are asymptotically normal. Thus, (\ref{Rr}) implies the asymptotic 
normality of $\Bigl\langle \hat P_r-{\mathbb E}\hat P_r u,v\Bigr\rangle ,u,v\in {\mathbb H}.$
This result will be discussed in detail in the next section.

The next statement immediately follows from Theorem \ref{technical_1}
and Bernstein inequality for sums of i.i.d. subexponential random variables 
$\xi_{j}(u,v), j=1,\dots, n.$ In particular, it shows that, under the assumptions of Theorem \ref{technical_1}, 
$$
\Bigl\langle \hat P_r-{\mathbb E}\hat P_r u,v\Bigr\rangle =O_{\mathbb P}(n^{-1/2})\ 
{\rm as}\ n\to\infty, u,v\in {\mathbb H}.
$$

\begin{corollary}
\label{concent_bilin}
Under the assumption of Theorem \ref{technical_1}, with some constants $D, D_{\gamma}>0,$ for all $u,v\in {\mathbb H}$ and for all $t\geq 1$ with probability at least $1-e^{-t},$
\begin{align}
\label{bdPr}
&
\Bigl|\Bigl\langle \hat P_r-{\mathbb E}\hat P_r u,v\Bigr\rangle\Bigr|
\leq D \frac{\|\Sigma\|_{\infty}}{\bar g_r}\sqrt{\frac{t}{n}}\|u\|\|v\|\notag\\
&\hspace{3cm}+
D_\gamma \frac{\|\Sigma\|_{\infty}^2}{\bar g_r^2}\biggl(\sqrt{\frac{{\bf r}(\Sigma)}{n}}\bigvee \sqrt{\frac{t}{n}}\bigvee \frac{t}{n}\biggr)
\sqrt{\frac{t}{n}}\|u\|\|v\|.
\end{align}
\end{corollary}

\begin{remark}
\label{rema_Cor1}
Note that $\xi_j(u,v)=0, j=1,\dots, n$ in the case when both $u$ and $v$ belong 
to the eigenspace corresponding to the eigenvalue $\mu_r$ (since, in this case,
$C_r u=C_rv=0$), or in the case when both $u$ and $v$ are in the orthogonal complement 
of this space (since then $P_ru=P_rv=0$). Therefore, for such $u,v$ the first term in the right hand side of (\ref{bdPr}) could be dropped and the bound reduces only to the second term.
\end{remark}

We now turn to the proof of Theorem \ref{technical_1}.

\begin{proof}
Clearly, it will be enough to prove bound (\ref{remaind}) for $\|u\|\leq 1, \|v\|\leq 1.$ This will be assumed throughout the proof.

First note that $L_r=L_r(E),$ where $E:=\hat \Sigma-\Sigma.$ Since 
$$
{\mathbb E}L_r={\mathbb E}L_r(E)=0,
$$
we get that
$$
R_r=L_r(E)+S_r(E)-{\mathbb E}(L_r(E)+S_r(E))-L_r(E)=S_r(E)-{\mathbb E}S_r(E)
$$
(recall Lemma \ref{lem-pert-spectral}).

Under condition (\ref{condition gamma_r-1}), we have $\mathbf{r}(\Sigma) \lesssim n$. Theorem \ref{spectrum_sharper} implies that
$$
\|\hat \Sigma - \Sigma\|_{\infty} \leq \E \|\hat \Sigma - \Sigma\|_{\infty} + C' \|\Sigma\|_{\infty} \left( \sqrt{\frac{t}{n}} \vee \frac{t}{n}  \right).
$$
If $C' \|\Sigma\|_{\infty} \left( \sqrt{\frac{t}{n}} \vee \frac{t}{n} \right) \leq \frac{\gamma \bar g_r }{4}$, then it is easy to see that $t\lesssim n$ and, for some $C>0$, 
$$
C' \|\Sigma\|_{\infty} \left( \sqrt{\frac{t}{n}} \vee \frac{t}{n} \right)  \leq C \|\Sigma\|_{\infty} \sqrt{\frac{t}{n}}.
$$
We will assume first that
\begin{align}\label{condition gamma_r-2}
C \|\Sigma\|_{\infty} \sqrt{\frac{t}{n}} \leq \frac{\gamma \bar g_r}{4}
\end{align}
(the proof of the concentration bound in the opposite case will be much easier).

Let 
\begin{align}\label{delta_n-t}
\delta_n(t):={\mathbb E}\|\hat \Sigma_n-\Sigma\|_{\infty}+C\|\Sigma\|_{\infty}
\sqrt{\frac{t}{n}}.
\end{align}
Theorem \ref{spectrum_sharper} gives that $\mathbb P \left\{ \|\hat \Sigma - \Sigma\|_{\infty} \geq \delta_n(t)\right\}\leq e^{-t}.$

The main part of the proof is the study of concentration of the random variable 
$\langle S_r(E)u,v\rangle$ around its expectation. To this end, 
we first study the concentration properties of ``truncated'' random variable  
$$
\Bigl\langle S_r(E)u,v\Bigr\rangle \varphi\Bigl(\frac{\|E\|_{\infty}}{\delta}\Bigr),
$$ 
where, for some $\gamma \in (0,1),$ $\varphi$ is a Lipschitz function with constant $\frac{1}{\gamma}$ on ${\mathbb R}_{+},$ $0\leq \varphi (s)\leq 1,$
$\varphi (s)=1, s\leq 1,$ $\varphi(s)=0, s>1+\gamma,$
and $\delta>0$ is such that $\|E\|_{\infty}\leq \delta$
with a high probability.

Our main tool is 
the following concentration inequality that easily follows 
from Gaussian isoperimetric inequality.

\begin{lemma}
\label{Gaussian_concentration}
Let $X_1,\dots, X_n$ be i.i.d. centered Gaussian random variables 
in ${\mathbb H}$ with covariance operator $\Sigma.$
Let $f:{\mathbb H}^n\mapsto {\mathbb R}$ be a function satisfying 
the following Lipschitz condition with some $L>0:$
$$
\Bigl|f(x_1,\dots, x_n)-f(x_1',\dots, x_n')\Bigr|\leq 
L\biggl(\sum_{j=1}^n \|x_j-x_j'\|^2\biggr)^{1/2},\ x_1,\dots, x_n,x_1',\dots,x_n'\in {\mathbb H}.
$$ 
Suppose that, for a real number $M,$
$$
{\mathbb P}\{f(X_1,\dots,X_n)\geq M\}\geq 1/4\ {\rm and}\ {\mathbb P}\{f(X_1,\dots,X_n)\leq M\}\geq 1/4.
$$
Then, there exists a numerical constant $D>0$ such that for all $t\geq 1,$ 
$$
{\mathbb P}\Bigl\{|f(X_1,\dots, X_n)-M|\geq DL\|\Sigma\|_{\infty}^{1/2}\sqrt{t}\Bigr\}\leq 
e^{-t}.
$$
\end{lemma}

Lemma \ref{Gaussian_concentration} will be applied to the function 
$$
f(X_1,\dots, X_n):=
\Bigl\langle S_r(E)u,v\Bigr\rangle \varphi\Bigl(\frac{\|E\|_{\infty}}{\delta}\Bigr)
$$
With a little abuse of notation, assume for now that $X_1,\dots, X_n$
are nonrandom vectors in ${\mathbb H}.$
For $X_1',\dots, X_n'\in {\mathbb H},$ denote 
$$
E'=\hat \Sigma'-\Sigma, \ \ 
\hat \Sigma'=n^{-1}\sum_{j=1}^n X_j'\otimes X_j'.
$$
Let $\hat P_r'$ be the orthogonal projector on the direct sum of 
eigenspaces of $\hat \Sigma'$ corresponding to its eigenvalues 
$\{\sigma_j(\hat \Sigma'):j\in \Delta_r\}.$

We have to check the Lipschitz condition for the function $f.$
We will start with the following simple fact based on perturbation theory 
bounds of Section \ref{sec:pert}.

\begin{lemma}
Let $\gamma \in (0,1)$ and suppose that 
\begin{equation}
\label{delt_le}
\delta \leq \frac{1-\gamma}{1+\gamma}\frac{\bar g_r}{2}.
\end{equation}
Suppose also that 
\begin{equation}
\label{E_le}
\|E\|_{\infty}\leq (1+\gamma)\delta\ {\rm and}\ \|E'\|_{\infty}\leq (1+\gamma)\delta.
\end{equation}
Then, there exists a constant $C_{\gamma}>0$ such that 
\begin{equation}
\label{Srlip}
\|S_r(E)-S_r(E')\|_{\infty}\leq C_{\gamma}\frac{\delta}{\bar g_r^2}\|E-E'\|_{\infty}.
\end{equation}
\end{lemma}
 
\begin{proof}
Note that, by the definition of $S_r(E),$
\begin{equation}
\label{one-1}
S_r(E')-S_r(E)=\hat P_r'-\hat P_r-L_r(E'-E).
\end{equation}
For $\hat P_r'-\hat P_r,$ we will use decomposition of Lemma \ref{lem-pert-spectral-general} that yields:
\begin{equation}
\label{two-2}
\hat P_r'-\hat P_r = \hat L_r(E'-E) + \hat S_r(E'-E)
\end{equation}
with 
$$
\hat L_r(E'-E)= 
\frac{1}{2\pi i} \oint_{\gamma_r}R_{\hat \Sigma}(\eta)(E'-E)R_{\hat \Sigma}(\eta)d\eta
$$
and 
\begin{equation}
\label{three-3}
\|\hat S_r(E'-E)\|_{\infty}\leq 
15\biggl(1+\frac{4}{\pi}\frac{\|E\|_{\infty}}{\bar g_r-2\|E\|_{\infty}}\biggr)
\frac{\|E-E'\|_{\infty}^2}{(\bar g_r-2\|E\|_{\infty})^2}.
\end{equation}
More precisely, we used Lemma \ref{lem-pert-spectral-general} with $\hat \Sigma$ instead of $\Sigma$ and with $\hat \Sigma'$ instead of $\tilde \Sigma.$ 
Observe that the set of eigenvalues $\{\sigma_j(\hat \Sigma):j\in \Delta_r\}$ can be written as $\{\mu_i(\hat \Sigma):i\in I\}$ for some $I\subset {\mathbb N}.$ Also, we have   $\Delta_I=\Delta_r,$ $\hat P_I=\hat P_r$ and $\hat P_I'=\hat P_r'.$ Finally,
in our case $L_I\leq 2\|E\|_{\infty}$ and 
$$
\bar g_I\geq \bar g_r-2\|E\|_{\infty}.
$$
We could also replace the contour $\gamma_I$ used in Lemma \ref{lem-pert-spectral-general} by the circle $\gamma_r$ since 
these two contours separate the same part of the spectrum of $\hat \Sigma$ from 
the rest of the spectrum. 

Note now that 
\begin{align*}
&
\hat L_r(E'-E)-L_r(E'-E)=
\\
&
\frac{1}{2\pi i} \oint_{\gamma_r}(R_{\hat \Sigma}(\eta)-R_{\Sigma}(\eta))(E'-E)R_{\hat \Sigma}(\eta)d\eta
+
\frac{1}{2\pi i} \oint_{\gamma_r}R_{\Sigma}(\eta)(E'-E)(R_{\hat \Sigma}(\eta)-R_{\Sigma}(\eta))d\eta,
\end{align*}
which implies the bound 
\begin{align}
\label{chetyre-4}
&
\|\hat L_r(E'-E)-L_r(E'-E)\|_{\infty}\leq 
\\
&
\nonumber
\biggl[\frac{1}{2\pi} 
\oint_{\gamma_r}\|R_{\hat \Sigma}(\eta)-R_{\Sigma}(\eta)\|_{\infty}\|R_{\hat \Sigma}(\eta)\|_{\infty}d\eta
+
\frac{1}{2\pi} \oint_{\gamma_r}\|R_{\Sigma}(\eta)\|_{\infty}\|R_{\hat \Sigma}(\eta)-R_{\Sigma}(\eta)\|_{\infty}d\eta\biggr]\|E-E'\|_{\infty}.
\end{align}
Since $\|E\|_{\infty}<\bar g_r/2,$ we get that, for all $\eta\in \gamma_r,$
$$
\|R_{\Sigma}(\eta)\|_{\infty}\leq \frac{2}{\bar g_r},\ \ 
\|R_{\hat \Sigma}(\eta)\|_{\infty}\leq \frac{2}{\bar g_r-2\|E\|_{\infty}}.
$$
Using respresentation (\ref{resolve}) (with $\hat \Sigma$ instead of $\tilde \Sigma$), we easily get that  
$$
\|R_{\hat \Sigma}(\eta)-R_{\Sigma}(\eta)\|_{\infty}
\leq \sum_{k\geq 1}\|R_{\Sigma}(\eta)\|_{\infty}^{k+1}\|E\|_{\infty}^k
\leq \frac{2}{\bar g_r}\frac{(2/\bar g_r)\|E\|_{\infty}}{1-(2/\bar g_r)\|E\|_{\infty}}= \frac{4\|E\|_{\infty}}{\bar g_r(\bar g_r-2\|E\|_{\infty})}.
$$
Due to these bounds, it follows from (\ref{chetyre-4}) that 
\begin{equation}
\label{pjat-5}
\|\hat L_r(E'-E)-L_r(E'-E)\|_{\infty}\leq
\frac{8\|E\|_{\infty}}{(\bar g_r-2\|E\|_{\infty})^2}\|E-E'\|_{\infty}.
\end{equation}
We combine now (\ref{one-1}), (\ref{two-2}), (\ref{three-3}) and (\ref{pjat-5})
to get 
\begin{align*}
&
\|S_r(E)-S_r(E')\|_{\infty}
\\
&
\leq \frac{8\|E\|_{\infty}}{(\bar g_r-2\|E\|_{\infty})^2}\|E-E'\|_{\infty}
+ 
15\biggl(1+\frac{4}{\pi}\frac{\|E\|_{\infty}}{\bar g_r-2\|E\|_{\infty}}\biggr)
\frac{\|E-E'\|_{\infty}^2}{(\bar g_r-2\|E\|_{\infty})^2}
\\
&
\frac{8\|E\|_{\infty}}{(\bar g_r-2\|E\|_{\infty})^2}\|E-E'\|_{\infty}
+ 
30\biggl(1+\frac{4}{\pi}\frac{\|E\|_{\infty}}{\bar g_r-2\|E\|_{\infty}}\biggr)
\frac{\|E\|_{\infty}\vee \|E'\|_{\infty}}{(\bar g_r-2\|E\|_{\infty})^2}\|E-E'\|_{\infty}.
\end{align*}
To complete the proof, it is enough to use conditions (\ref{delt_le}), (\ref{E_le})
that, in particular, imply
$$
\bar g_r-2\|E\|_{\infty}\geq \bar g_r-2(1+\gamma) \delta \geq \gamma \bar g_r.
$$

\qed
\end{proof}

\begin{lemma}
\label{Lipschitz_constant}
Suppose that, for some $\gamma\in (0,1/2),$
\begin{equation}
\label{con_delta}
\delta \leq \frac{1-2\gamma}{1+2\gamma}\frac{\bar g_r}{2}.
\end{equation}
Then, there exists a constant $D_{\gamma}>0$ such that, for all $X_1,\dots, X_n,
X_1',\dots, X_n'\in {\mathbb H},$
\begin{equation}
\label{lip_lip_CC}
|f(X_1,\dots, X_n)-f(X_1',\dots, X_n')| 
\leq 
D_{\gamma} \frac{\delta}{\bar g_r^2}
\frac{\|\Sigma\|_{\infty}^{1/2}+\delta^{1/2}}{\sqrt{n}}
\biggl(\sum_{j=1}^n \|X_j-X_j'\|^2\biggr)^{1/2}.
\end{equation}
\end{lemma}

\begin{proof}
Since $\varphi \biggl(\frac{\|E\|_{\infty}}{\delta}\biggr)=0$
if $\|E\|_{\infty}\geq (1+\gamma)\delta,$ bound (\ref{remainder_A}) of Lemma 
\ref{lem-pert-spectral} implies that 
\begin{equation}
\label{bdf}
|f(X_1,\dots, X_n)|= \biggl|\langle S_r(E)u,v\rangle \varphi \biggl(\frac{\|E\|_{\infty}}{\delta}\biggr)\biggr|
\leq 14(1+\gamma)^2 \frac{\delta^2}{\bar g_r^2}.
\end{equation}
Using now bounds (\ref{Srlip}), (\ref{bdf}) and the fact that $\varphi$ bounded by $1$ and Lipschitz with 
constant $\frac{1}{\gamma},$ which implies that the function $t\mapsto \varphi\biggl(\frac{t}{\delta}\biggr)$ is Lipschitz with constant $\frac{1}{\gamma \delta},$ 
we easily get that, under the assumptions 
\begin{equation}
\label{assumEE}
\|E\|_{\infty}\leq (1+\gamma)\delta,\ \|E'\|_{\infty}\leq (1+\gamma)\delta, 
\end{equation}
the following inequality holds:
\begin{align}
\label{bdassEE}
&
\biggl|\Bigl\langle S_r(E)u,v\Bigr\rangle \varphi \biggl(\frac{\|E\|_{\infty}}{\delta}\biggr)-
\Bigl\langle S_r(E')u,v\Bigr\rangle \varphi \biggl(\frac{\|E'\|_{\infty}}{\delta}\biggr)\biggr|
\\
&
\nonumber
\leq \|S_r(E)-S_r(E')\|_{\infty}+\frac{14(1+\gamma)^2}{\gamma}
\frac{\delta}{\bar g_r^2}\|E-E'\|_{\infty}
\\
&
\nonumber
\leq \biggl(C_{\gamma}+\frac{14(1+\gamma)^2}{\gamma}\biggr)\frac{\delta}{\bar g_r^2}\|E-E'\|_{\infty}.
\end{align}
It remains to prove a similar bound in the case when 
$$
\|E\|_{\infty}\leq (1+\gamma)\delta,\ \|E'\|_{\infty}> (1+\gamma)\delta
$$ 
(when both norms are larger than $(1+\gamma)\delta,$ the function $\varphi$ is equal to zero and the bound is trivial). First consider the case when $\|E-E'\|_{\infty}\geq \gamma \delta.$ Then, in view of (\ref{bdf}), we have 
\begin{align*}
&
\biggl|\langle S_r(E)u,v\rangle \varphi \biggl(\frac{\|E\|_{\infty}}{\delta}\biggr)-
\langle S_r(E')u,v\rangle \varphi \biggl(\frac{\|E\|_{\infty}}{\delta}\biggr)\biggr|
\\
&
=\biggl|\langle S_r(E)u,v\rangle \varphi \biggl(\frac{\|E\|_{\infty}}{\delta}\biggr)\biggr|\leq 
14(1+\gamma)^2 \frac{\delta^2}{\bar g_r^2} \leq 
\frac{14(1+\gamma)^2}{\gamma} \frac{\delta}{\bar g_r^2}\|E-E'\|_{\infty}.
\end{align*}
Finally, if $\|E-E'\|_{\infty}<\gamma \delta,$ we have that $\|E'\|_{\infty}\leq  (1+2\gamma)\delta$ and, taking into account assumption (\ref{con_delta}), we can repeat the argument in the case (\ref{assumEE}) ending up with the same bound as (\ref{bdassEE}) with constant $C_{2\gamma}+\frac{14(1+2\gamma)^2}{\gamma}$ 
instead of $C_{\gamma}+\frac{14(1+\gamma)^2}{\gamma}$ in the right hand side.
Thus, with some constant $L_{\gamma}>0,$
\begin{equation}
\label{bdEE'}
\biggl|\langle S_r(E)u,v\rangle \varphi \biggl(\frac{\|E\|_{\infty}}{\delta}\biggr)-
\langle S_r(E')u,v\rangle \varphi \biggl(\frac{\|E\|_{\infty}}{\delta}\biggr)\biggr|
\leq 
L_{\gamma}\frac{\delta}{\bar g_r^2}\|E-E'\|_{\infty}.
\end{equation}

We will now control $\|E-E'\|_{\infty}.$ Note that 
$$
\|E-E'\|_{\infty}= 
\sup_{\|u\|\leq 1,\|v\|\leq 1}
\Bigl|\Bigl\langle (E-E')u,v\Bigr\rangle\Bigr|
$$
$$
=
\sup_{\|u\|\leq 1,\|v\|\leq 1}
\biggl|n^{-1}\sum_{j=1}^n 
\langle X_j,u\rangle \langle X_j,v\rangle
-\langle X_j',u\rangle \langle X_j',v\rangle
\biggr|
$$
$$
\leq 
\sup_{\|u\|\leq 1,\|v\|\leq 1}
\biggl|n^{-1}\sum_{j=1}^n 
\langle X_j,u\rangle \langle X_j-X_j',v\rangle
\biggr|+
\sup_{\|u\|\leq 1,\|v\|\leq 1}
\biggl|n^{-1}\sum_{j=1}^n 
\langle X_j-X_j',u\rangle \langle X_j',v\rangle
\biggr|
$$
$$
\leq 
\sup_{\|u\|\leq 1}
\biggl(n^{-1}\sum_{j=1}^n 
\langle X_j,u\rangle^2
\biggr)^{1/2} 
\sup_{\|v\|\leq 1}
\biggl(n^{-1}\sum_{j=1}^n\langle X_j-X_j',v\rangle^2\biggr)^{1/2}
$$
$$
+
\sup_{\|u\|\leq 1}
\biggl(n^{-1}\sum_{j=1}^n 
\langle X_j-X_j',u\rangle^2
\biggr)^{1/2} 
\sup_{\|v\|\leq 1}
\biggl(n^{-1}\sum_{j=1}^n\langle X_j',v\rangle^2\biggr)^{1/2}
$$
$$
\leq \frac{\|\hat \Sigma\|_{\infty}^{1/2}+
\|\hat \Sigma'\|_{\infty}^{1/2}}{\sqrt{n}}
\biggl(\sum_{j=1}^n \|X_j-X_j'\|^2\biggr)^{1/2}.
$$
Clearly, it is enough to consider the case when at least one of the 
norms $\|E\|_{\infty}, \|E'\|_{\infty}$ is not larger than $2\delta.$
To be specific, assume that $\|E\|_{\infty}\leq 2\delta.$ Then 
$$
\|\hat \Sigma\|_{\infty}^{1/2}+\|\hat \Sigma'\|_{\infty}^{1/2}
\leq 
2\|\hat \Sigma\|_{\infty}^{1/2}+\|E-E'\|_{\infty}^{1/2}
\leq 
2\|\Sigma\|_{\infty}^{1/2}+2\sqrt{2\delta}+\|E-E'\|_{\infty}^{1/2}.
$$
Therefore,
$$
\|E-E'\|_{\infty}\leq 
\frac{2\|\Sigma\|_{\infty}^{1/2}+2\sqrt{2\delta}}{\sqrt{n}}
\biggl(\sum_{j=1}^n \|X_j-X_j'\|^2\biggr)^{1/2}
+
\frac{\|E-E'\|_{\infty}^{1/2}}{\sqrt{n}}
\biggl(\sum_{j=1}^n \|X_j-X_j'\|^2\biggr)^{1/2},
$$
which easily implies 
\begin{equation}
\label{E__E'}
\|E-E'\|_{\infty}\leq 
\frac{4\|\Sigma\|_{\infty}^{1/2}+4\sqrt{2\delta}}{\sqrt{n}}
\biggl(\sum_{j=1}^n \|X_j-X_j'\|^2\biggr)^{1/2}
\bigvee
\frac{4}{n}
\sum_{j=1}^n \|X_j-X_j'\|^2.
\end{equation}
Now substitute the last bound in the right hand side of (\ref{bdEE'}) and also 
observe that, in view of (\ref{bdf}), the left hand side of (\ref{bdEE'}) can be 
also upper bounded by 
$
28(1+\gamma)^2 \frac{\delta^2}{\bar g_r^2}.
$
Therefore, we get that with some constant $L_{\gamma}'>0,$
\begin{align}
\label{bdEE''}
&
\biggl|\langle S_r(E)u,v\rangle \varphi \biggl(\frac{\|E\|_{\infty}}{\delta}\biggr)-
\langle S_r(E')u,v\rangle \varphi \biggl(\frac{\|E\|_{\infty}}{\delta}\biggr)\biggr|
\\
&
\nonumber
\leq 4L_{\gamma}\frac{\delta}{\bar g_r^2}
\biggl[
\frac{\|\Sigma\|_{\infty}^{1/2}+\sqrt{2\delta}}{\sqrt{n}}
\biggl(\sum_{j=1}^n \|X_j-X_j'\|^2\biggr)^{1/2}
\bigvee
\frac{1}{n}
\sum_{j=1}^n \|X_j-X_j'\|^2
\biggr]
\bigwedge 
28(1+\gamma)^2 \frac{\delta^2}{\bar g_r^2}
\\
&
\nonumber
\leq L_{\gamma}'\frac{\delta}{\bar g_r^2}\biggl[
\frac{\|\Sigma\|_{\infty}^{1/2}+\sqrt{2\delta}}{\sqrt{n}}
\biggl(\sum_{j=1}^n \|X_j-X_j'\|^2\biggr)^{1/2}
\bigvee
\biggl(\frac{1}{n}
\sum_{j=1}^n \|X_j-X_j'\|^2\bigwedge \delta 
\biggr)\biggr].
\end{align}
Using an elementary inequality $a\wedge b\leq \sqrt{ab}, a,b\geq 0,$
we get 
$$
\frac{1}{n}
\sum_{j=1}^n \|X_j-X_j'\|^2\bigwedge \delta \leq 
\sqrt{\frac{\delta}{n}}\biggl(\sum_{j=1}^n \|X_j-X_j'\|^2\biggr)^{1/2}.
$$
This allows us to drop the last term in the maximum in the right hand side of 
(\ref{bdEE''}) (since a similar expression is a part of the first term). 
This yields bound (\ref{lip_lip_CC}).

\qed
\end{proof}

We set $\delta :=\delta_n(t)$. Without loss of generality, we can assume that $t\geq \log 4$ and $e^{-t}\leq 1/4$
(the result can be extended to all $t\geq 1$ by adjusting the constants). Theorem \ref{spectrum_sharper} gives $\mathbb P \left(  \|E\|_{\infty} \geq \delta \right) \leq \frac{1}{4}$. In addition, in view of (\ref{condition gamma_r-1}) and (\ref{condition gamma_r-2}), $\delta_n(t) \leq \left(1-\frac{\gamma}{2}\right)\frac{\bar g_r}{2} = \frac{1-2\gamma'}{1+2\gamma'}\frac{\bar g_r}{2} $ for some $\gamma' \in (0,1/2)$. Thus the function $f(X_1,\ldots,X_n)$ satisfies the Lipschitz condition (\ref{lip_lip_CC}) with some constant $D_{\gamma}' = D_{\gamma'}$.

To complete the proof of Theorem \ref{technical_1}, denote ${\rm Med}(\eta)$ a median of a random variable $\eta,$ and let 
$M:={\rm Med}\Bigl(\Bigl\langle S_r(E)u,v\Bigr\rangle\Bigr).$ 
Since $f(X_1,\dots, X_n)=\langle S_r(E)u,v\rangle$ on the event $\{\|E\|_{\infty}<\delta\},$ we have 
$$
{\mathbb P}\{f(X_1,\dots, X_n)\geq M\}\geq 
{\mathbb P}\{f(X_1,\dots, X_n)\geq M, \|E\|_{\infty}<\delta\}= 
$$
$$
{\mathbb P}\{\langle S_r(E)u,v\rangle\geq M, \|E\|_{\infty}<\delta\}\geq 
{\mathbb P}\{\langle S_r(E)u,v\rangle\geq M\}-{\mathbb P}\{\|E\|_{\infty}\geq \delta\}
\geq 1/4
$$
and, similarly,
$$
{\mathbb P}\{f(X_1,\dots, X_n)\leq M\}\geq 1/4.
$$
It follows from Lemma \ref{Gaussian_concentration} and Lemma \ref{Lipschitz_constant} that with some constant $D_{\gamma}>0,$ for all $t\geq 1$ with probability at least $1-e^{-t},$
$$
|f(X_1,\dots, X_n)-M|\leq D_{\gamma} \frac{\delta}{\bar g_r^2}
(\|\Sigma\|_{\infty}^{1/2}+\delta^{1/2})
\|\Sigma\|_{\infty}^{1/2} \sqrt{\frac{t}{n}}.
$$
In this case, we get for all $t\geq 1$, with probability at least $1-2e^{-t}$
\begin{align*}
\Bigl|\Bigl\langle S_r(E)u,v\Bigr\rangle-M\Bigr|&\leq 
C_{\gamma}\frac{\|\Sigma\|_{\infty}^2}{\bar g_r^2} \left(\sqrt{\frac{\mathbf{r}(\Sigma)}{n}}\vee \sqrt{\frac{t}{n}} \right)\sqrt{\frac{t}{n}}\\
&\leq 
C_{\gamma}\frac{\|\Sigma\|_{\infty}^2}{\bar g_r^2} \left(\sqrt{\frac{\mathbf{r}(\Sigma)}{n}}\vee \sqrt{\frac{t}{n}}\vee\frac{t}{n} \right)\sqrt{\frac{t}{n}}.
\end{align*}
for some constant $C_\gamma >0$. By adjusting the constant $C_\gamma$, we can replace $1-2e^{-t}$ by $1-e^{-t}$.

We will now prove a similar bound in the case where (\ref{condition gamma_r-2}) does not hold. Then,
\begin{align}\label{condition_gamma_r-3}
\frac{\|\Sigma\|_{\infty}}{\bar g_r} \sqrt{\frac{t}{n}} \geq \frac{\gamma }{4C}.
\end{align}
If follows from bound (\ref{remainder_A}) and condition $\|E\|_{\infty} \leq \bar g_r /2$ that
$$
\Bigl|\Bigl\langle S_r(E)u,v\Bigr\rangle\Bigr|\leq \|S_r(E)\|_{\infty} \leq c \frac{\|E\|_{\infty}}{\bar g_r}.
$$
We can now use the bounds of Theorems \ref{th_operator} and \ref{spectrum_sharper} combined with (\ref{condition_gamma_r-3}) to get for all $t\geq 1$, with probability at least $1-e^{-t}$ that
$$
\Bigl|\Bigl\langle S_r(E)u,v\Bigr\rangle\Bigr| \leq c' \frac{\|\Sigma\|_{\infty}}{\bar g_r} \left( \sqrt{\frac{\mathbf{r}(\Sigma)}{n}} \vee \sqrt{\frac{t}{n}} \vee \frac{t}{n}\right) \leq C_{\gamma}' \frac{\|\Sigma\|_{\infty}^2}{\bar g_r^2} \left( \sqrt{\frac{\mathbf{r}(\Sigma)}{n}} \vee \sqrt{\frac{t}{n}} \vee \frac{t}{n}\right)\sqrt{\frac{t}{n}},
$$
for some $C_{\gamma}'>0$.

It is easy to deduce from the previous display that
$$
M \leq C_{\gamma}' \frac{\|\Sigma\|_{\infty}^2}{\bar g_r^2} \left( \sqrt{\frac{\mathbf{r}(\Sigma)}{n}} \vee \sqrt{\frac{1}{n}} \vee \frac{1}{n}\right)\sqrt{\frac{1}{n}}.
$$
Combining the last two displays, we get for all $t\geq 1$, with probability at least $1-e^{-t}$
$$
\Bigl|\Bigl\langle S_r(E)u,v\Bigr\rangle - M \Bigr| \leq C_{\gamma}' \frac{\|\Sigma\|_{\infty}^2}{\bar g_r^2} \left( \sqrt{\frac{\mathbf{r}(\Sigma)}{n}} \vee \sqrt{\frac{t}{n}} \vee \frac{t}{n}\right)\sqrt{\frac{t}{n}},
$$
By integrating the tails of this exponential bound it is easy to see that,
with some $D_{\gamma}>0,$ 
$$
\Bigl|{\mathbb E}\Bigl\langle S_r(E)u,v\Bigr\rangle-M\Bigr|\leq 
{\mathbb E}\Bigl|\Bigl\langle S_r(E)u,v\Bigr\rangle-M\Bigr|
\leq 
D_{\gamma}\frac{\|\Sigma\|_{\infty}^2}{\bar g_r^2} \left( \sqrt{\frac{\mathbf{r}(\Sigma)}{n}} \vee \sqrt{\frac{1}{n}} \vee \frac{1}{n}\right)\sqrt{\frac{1}{n}},
$$
which, in turn, implies that one can replace $M$ by the expectation 
${\mathbb E}\Bigl\langle S_r(E)u,v\Bigr\rangle$ in the concentration bound and 
get that with some $D_{\gamma}>0$ and with probability at least $1-2e^{-t}$
\begin{align*}
&
\Bigl|\Bigl\langle S_r(E)u,v\Bigr\rangle-{\mathbb E}\Bigl\langle S_r(E)u,v\Bigr\rangle\Bigr|\leq 
D_{\gamma} \frac{\|\Sigma\|_{\infty}^2}{\bar g_r^2} \left( \sqrt{\frac{\mathbf{r}(\Sigma)}{n}} \vee \sqrt{\frac{t}{n}} \vee \frac{t}{n}\right)\sqrt{\frac{t}{n}}.
\end{align*}
This completes the proof of the theorem.
%

\qed
\end{proof}

\section{A Representation of the Bias ${\mathbb E}\hat P_r-P_r$}\label{Sec:Bias}

In this section, we study the bias ${\mathbb E}\hat P_r-P_r$ of the empirical
spectral projector $\hat P_r.$ Under mild assumptions, we show that   
$$
{\mathbb E}\hat P_r-P_r=P_rW_rP_r +T_r,
$$
where the main term $P_rW_rP_r$ is a symmetric operator of rank $m_r$ such that 
\begin{equation}
\label{W_r}
\|P_rW_rP_r\|_{\infty} \leq \|W_r\|_{\infty}\lesssim \frac{\|\Sigma\|_{\infty}^2}{\bar g_r^2}\frac{{\bf r}(\Sigma)}{n}
\end{equation}
and the remainder term $T_r$ satisfies the condition 
$\|T_r\|_{\infty}=O(n^{-1/2}).$ Moreover, in the case when 
${\bf r}(\Sigma)=o(n),$ we have $\|T_r\|_{\infty}=o(n^{-1/2}).$

\begin{theorem}
\label{bd_bias}
Suppose that, for some $\gamma\in (0,1),$ (\ref{condition gamma_r-1}) is satisfied. Denote $W_r:={\mathbb E}S_r(\hat \Sigma-\Sigma).$
Then, there exists a constant $D_{\gamma}>0$ such that
\begin{equation}
\label{T_r}
\Bigl\|{\mathbb E}\hat P_r-P_r-P_rW_rP_r\Bigr\|_{\infty}\leq D_{\gamma}\frac{m_r\|\Sigma\|_{\infty}^2}{\bar g_r^2}
\sqrt{\frac{{\bf r}(\Sigma)}{n}}\frac{1}{\sqrt{n}}.
\end{equation}
\end{theorem}

\begin{remark}
Note that the operator $P_rW_rP_r$ does satisfy condition (\ref{W_r}). 
This follows from bound (\ref{remainder_A}) and Theorem \ref{th_operator}.
\end{remark}

\begin{proof}
We set
$
\delta_n := \E\|E\|_{\infty}$. We recall that, in view of Condition (\ref{condition gamma_r-1}) and Theorem \ref{th_operator}, we have $\mathbf{r}(\Sigma)\lesssim n$ and $\delta_n < \frac{\bar g_r}{2}$. We start with the following representation
\begin{align}
\nonumber
&
{\mathbb E}\hat P_r-P_r= {\mathbb E}(L_r(E)+S_r(E))=
{\mathbb E}S_r(E)
={\mathbb E}P_r S_r(E) P_r 
\\
&
\label{SrPrE}
+ 
{\mathbb E}\Bigl(P_r^{\perp}S_r(E)P_r+P_r S_r(E)P_r^{\perp}+
P_r^{\perp}S_r(E)P_r^{\perp}\Bigr)I(\|E\|_{\infty}\leq \delta_n)
\\
&
\nonumber
+{\mathbb E}\Bigl(P_r^{\perp}S_r(E)P_r+P_r S_r(E)P_r^{\perp}+
P_r^{\perp}S_r(E)P_r^{\perp}\Bigr)I(\|E\|_{\infty}>\delta_n).
\end{align}
and provide bounds for its relevant terms. 

Recall formula (\ref{resolve}) and note that, under the assumption $\|E\|_{\infty}
<\frac{\bar g_r}{2},$  the series in the right hand side 
converges in the operator norm absolutely and uniformly in $\eta\in \gamma_r.$ 
Under this assumption, 
\begin{equation}
\label{srE2}
S_r(E)=
- \sum_{k\geq 2}
\frac{1}{2\pi i} \oint_{\gamma_r}(-1)^{k}[R_{\Sigma}(\eta)E]^kR_{\Sigma}(\eta)d\eta. 
\end{equation}
Denote 
$$
\tilde R_{\Sigma}(\eta):= \sum_{s\not\in \Delta_r} \frac{1}{\mu_s-\eta}P_s.
$$
Then 
$$
R_{\Sigma}(\eta)= \frac{1}{\mu_r-\eta}P_r + \tilde R_{\Sigma}(\eta). 
$$
It is easy to check that 
$$
P_r^{\perp}[R_{\Sigma}(\eta)E]^kR_{\Sigma}(\eta)P_r = 
P_r^{\perp}\frac{1}{\mu_r-\eta}[R_{\Sigma}(\eta)E]^kP_r
$$
$$
=
\frac{1}{(\mu_r-\eta)^2}
\sum_{s=2}^{k} 
(\tilde R_{\Sigma}(\eta)E)^{s-1}P_rE (R_{\Sigma}(\eta)E)^{k-s}P_r 
+ \frac{1}{\mu_r-\eta}(\tilde R_{\Sigma}(\eta)E)^k P_r
$$
To understand the last equality, note that, in each bracket of the expression 
$$
[R_{\Sigma}(\eta)E]^k = [R_{\Sigma}(\eta)E]\dots [R_{\Sigma}(\eta)E],
$$ 
$R_{\Sigma}(\eta)$ can be replaced 
by the sum of two terms, $\frac{1}{\mu_r-\eta}P_r$ and $\tilde R_{\Sigma}(\eta).$
Index $s$ in the sum is the number of the first bracket where $\frac{1}{\mu_r-\eta}P_r$ is chosen. If $s=1,$ the corresponding term
is equal to $0$ since $P_r^{\perp}P_r=0.$
The last term corresponds to the case when 
$\tilde R_{\Sigma}(\eta)$ is chosen from each of the brackets. 

We can now write 
\begin{align}
\nonumber
&
P_r^{\perp}S_r(E)P_r =
\\
& 
\label{Prperp}
- \sum_{k\geq 2}(-1)^k\frac{1}{2\pi i}\oint_{\gamma_r} \biggl[\frac{1}{(\mu_r-\eta)^2}\sum_{s=2}^{k} (\tilde R_{\Sigma}(\eta)E)^{s-1}P_rE (R_{\Sigma}(\eta)E)^{k-s}P_r +
\\
&
\nonumber 
\frac{1}{\mu_r-\eta}(\tilde R_{\Sigma}(\eta)E)^k P_r\biggr]d\eta.
\end{align}
Since $P_r=\sum_{l\in \Delta_r}(\theta_l\otimes \theta_l),$ 
where $\{\theta_l:l\in \Delta_r\}$ is an arbitrary orthonormal basis of the eigenspace corresponding to 
the eigenvalue $\mu_r,$
we get that, 
for all $v\in {\mathbb H},$ 
\begin{align}
&
(\tilde R_{\Sigma}(\eta)E)^{s-1}P_rE (R_{\Sigma}(\eta)E)^{k-s}P_rv
=\sum_{l\in \Delta_r}
(\tilde R_{\Sigma}(\eta)E)^{s-1}(\theta_l\otimes \theta_l)
E (R_{\Sigma}(\eta)E)^{k-s}P_rv=
\\
&
\label{qua0}
=\sum_{l\in \Delta_r}
\Bigl\langle E (R_{\Sigma}(\eta)E)^{k-s}P_rv,\theta_l\Bigr\rangle 
(\tilde R_{\Sigma}(\eta)E)^{s-2}\tilde R_{\Sigma}(\eta)E\theta_l.
\end{align}
Clearly,
\begin{equation}
\nonumber
\Bigl|\Bigl\langle E (R_{\Sigma}(\eta)E)^{k-s}P_rv,\theta_l\Bigr\rangle\Bigr|
\leq \|R_{\Sigma}(\eta)\|_{\infty}^{k-s} \|E\|_{\infty}^{k-s+1} \|v\|,
\end{equation}
which implies that 
\begin{equation}
\label{qua1}
{\mathbb E}\Bigl|\Bigl\langle E (R_{\Sigma}(\eta)E)^{k-s}P_rv,\theta_l\Bigr\rangle\Bigr|^2
I(\|E\|_{\infty}\leq \delta_n)
\leq 
\biggl(\frac{2}{\bar g_r}\biggr)^{2(k-s)}\delta_n^{2(k-s+1)}\|v\|^2.
\end{equation}
We also have 
$$
(\tilde R_{\Sigma}(\eta)E)^{s-2}\tilde R_{\Sigma}(\eta)E\theta_l=
(\tilde R_{\Sigma}(\eta)E)^{s-2}\tilde R_{\Sigma}(\eta)(\hat \Sigma-\Sigma)\theta_l
$$
$$
=
(\tilde R_{\Sigma}(\eta)E)^{s-2}\tilde R_{\Sigma}(\eta)\hat \Sigma\theta_l
=
n^{-1}\sum_{j=1}^n \langle X_j,\theta_l\rangle 
(\tilde R_{\Sigma}(\eta)E)^{s-2}\tilde R_{\Sigma}(\eta)X_j,
$$
where we used the fact that 
$$
\tilde R_{\Sigma}(\eta) \Sigma\theta_l=
\mu_r\tilde R_{\Sigma}(\eta)\theta_l=0.
$$
It is easy to check that the random variables 
$(\tilde R_{\Sigma}(\eta)E)^{s-2}\tilde R_{\Sigma}(\eta)X_j, j=1,\dots, n$
are functions of random variables $P_s X_j: s\neq r, j=1,\dots, n$ 
that are independent of $\langle X_j,\theta_l\rangle, l\in \Delta_r, j=1,\dots, n$ (recall 
that $X_j, j=1,\dots, n$ are i.i.d. Gaussian, and $P_r X_j, j=1,\dots, n$ and 
$P_s X_j: s\neq r, j=1,\dots, n$ are uncorrelated and, hence, independent). Given $u\in {\mathbb H},$ denote 
$$
\zeta_j(u)=
\Bigl\langle (\tilde R_{\Sigma}(\eta)E)^{s-2}\tilde R_{\Sigma}(\eta)X_j,u\Bigr\rangle,\ j=1,\dots, n,
$$
(which are complex valued random variables). 
Write $\zeta_j(u)=\zeta_j^{(1)}(u)+i \zeta_j^{(2)}(u),$
where $\zeta_j^{(1)}(u), \zeta_j^{(2)}(u)$ are real valued.
Denote also 
$$
\alpha (u) := \alpha^{(1)}(u)+i\alpha^{(2)}(u):= 
\Bigl\langle((\tilde R_{\Sigma}(\eta)E)^{s-2}\tilde R_{\Sigma}(\eta)E\theta_l,u\Bigr\rangle.
$$
Then, conditionally on $P_s X_j: s\neq r, j=1,\dots, n,$
the random vector $(\alpha^{(1)}(u), \alpha^{(2)}(u))$ has the same 
distribution as mean zero Gaussian random vector in ${\mathbb R}^2$
with covariance 
$$
\frac{\mu_r}{n} \biggl(n^{-1}\sum_{j=1}^n \zeta_j^{(k_1)}(u)\zeta_j^{(k_2)}(u)\biggr),\ k_1,k_2=1,2.
$$ 
Note that 
$$
n^{-1}\sum_{j=1}^n |\zeta_j(u)|^2 = n^{-1}\sum_{j=1}^n \Bigl|\Bigl\langle (\tilde R_{\Sigma}(\eta)E)^{s-2}\tilde R_{\Sigma}(\eta)X_j,u\Bigr\rangle\Bigr|^2
$$
$$
=\Bigl\langle \hat \Sigma 
(\tilde R_{\Sigma}(\eta)E)^{s-2}\tilde R_{\Sigma}(\eta)u,
(\tilde R_{\Sigma}(\eta)E)^{s-2}\tilde R_{\Sigma}(\eta)u
\Bigr\rangle
\leq 
\|\hat \Sigma\|_{\infty} \|\tilde R_{\Sigma}(\eta)\|_{\infty}^{2(s-1)}\|E\|_{\infty}^{2(s-2)}\|u\|^2
$$
$$
\leq 
\biggl(\|\Sigma\|_{\infty} 
\|\tilde R_{\Sigma}(\eta)\|_{\infty}^{2(s-1)}\|E\|_{\infty}^{2(s-2)}
+\|\tilde R_{\Sigma}(\eta)\|_{\infty}^{2(s-1)}\|E\|_{\infty}^{2s-3}\biggr)\|u\|^2.
$$
Under the assumption
$\delta_n<\frac{\bar g_r}{2},$ the following inclusion holds:
$$
\Bigl\{\|E\|_{\infty}\leq \delta_n\Bigr\}\subset 
\biggl\{n^{-1}\sum_{j=1}^n |\zeta_j(u)|^2 \leq 
2\|\Sigma\|_{\infty} \biggl(\frac{2}{\bar g_r}\biggr)^{2(s-1)}\delta_n^{2(s-2)}\|u\|^2\biggr\}=: G.
$$
Therefore, we have 
\begin{align}
\label{qua2}
&
\nonumber
{\mathbb E}\Bigl|\Bigl\langle( \tilde R_{\Sigma}(\eta)E)^{s-2}\tilde R_{\Sigma}(\eta)E\theta_l,u\Bigr\rangle\Bigr|^2 I(\|E\|_{\infty}\leq \delta_n)
\leq  
{\mathbb E}\Bigl|\Bigl\langle(\tilde R_{\Sigma}(\eta)E)^{s-2}\tilde R_{\Sigma}(\eta)E\theta_l,u\Bigr\rangle\Bigr|^2 
I_G
\\
&
\nonumber
={\mathbb E}{\mathbb E}\biggl(\Bigl|\Bigl\langle((\tilde R_{\Sigma}(\eta)E)^{s-2}\tilde R_{\Sigma}(\eta)E\theta_l,u\Bigr\rangle\Bigr|^2 
I_G\Bigl| P_s X_j, s\neq r, j=1,\dots, n\biggr) 
\\
&
\nonumber
=
\frac{\mu_r}{n}{\mathbb E}{\mathbb E}\biggl(n^{-1}\sum_{j=1}^n |\zeta_j(u)|^2
I_G\Bigl| P_s X_j, s\neq r, j=1,\dots, n
\biggr)
\\
&
=
\frac{\mu_r}{n}{\mathbb E}n^{-1}\sum_{j=1}^n |\zeta_j(u)|^2
I_G
\leq 
2\|\Sigma\|_{\infty}\frac{\mu_r}{n}\biggl(\frac{2}{\bar g_r}\biggr)^{2(s-1)}\delta_n^{2(s-2)} \|u\|^2.
\end{align}
By (\ref{qua1}) and (\ref{qua2}),
\begin{align}
\label{qua_qua_1}
&
\nonumber
\Bigl|
{\mathbb E}\Bigl\langle E (R_{\Sigma}(\eta)E)^{k-s}P_rv,\theta_l\Bigr\rangle 
\Bigl\langle (\tilde R_{\Sigma}(\eta)E)^{s-2}\tilde R_{\Sigma}(\eta)E\theta_l,u\Bigr\rangle
I(\|E\|_{\infty}\leq \delta_n)\Bigr|
\\
&
\nonumber 
\leq \biggl({\mathbb E}\Bigl|\Bigl\langle E (R_{\Sigma}(\eta)E)^{k-s}P_rv,\theta_l\Bigr\rangle\Bigr|^2
I(\|E\|_{\infty}\leq \delta_n)\biggr)^{1/2}
\biggl({\mathbb E}\Bigl|\Bigl\langle((\tilde R_{\Sigma}(\eta)E)^{s-2}\tilde R_{\Sigma}(\eta)E\theta_l,u\Bigr\rangle\Bigr|^2 I(\|E\|_{\infty}\leq \delta_n)\biggr)^{1/2}
\\
&
\leq \sqrt{2}\frac{\|\Sigma\|_{\infty}}{\sqrt{n}}\biggl(\frac{2\delta_n}{\bar g_r}\biggr)^{k-1}\|u\|\|v\|
\end{align}
and it follows from (\ref{qua0}) and (\ref{qua_qua_1}) that 
\begin{align}
\label{qua_qua_2}
&
\biggl|{\mathbb E}\Bigl\langle (\tilde R_{\Sigma}(\eta)E)^{s-1}P_rE (R_{\Sigma}(\eta)E)^{k-s}P_rv,u\Bigr\rangle I(\|E\|_{\infty}\leq \delta_n)
\biggr|
\\
&
\nonumber
\leq 
\sum_{l\in \Delta_r}
\biggl|
{\mathbb E}\Bigl\langle E (R_{\Sigma}(\eta)E)^{k-s}P_rv,\theta_l\Bigr\rangle 
\Bigl\langle (\tilde R_{\Sigma}(\eta)E)^{s-2}\tilde R_{\Sigma}(\eta)E\theta_l,u\Bigr\rangle
I(\|E\|_{\infty}\leq \delta_n)
\biggr| 
\\
&
\nonumber
\leq 
\sqrt{2}m_r\frac{\|\Sigma\|_{\infty}}{\sqrt{n}}\biggl(\frac{2\delta_n}{\bar g_r}\biggr)^{k-1}\|u\|\|v\|.
\end{align}
Similarly, we also have 
\begin{equation}
\label{qua_qua_3}
\Bigl|{\mathbb E}\Bigl\langle (\tilde R_{\Sigma}(\eta)E)^k P_rv,u\Bigr\rangle\Bigr| \leq 
\sqrt{2}m_r\frac{\|\Sigma\|_{\infty}}{\sqrt{n}}\frac{2}{\bar g_r}\biggl(\frac{2\delta_n}{\bar g_r}\biggr)^{k-1}\|u\|\|v\|.
\end{equation}

Now use (\ref{Prperp}), (\ref{qua_qua_2}) and (\ref{qua_qua_3}) to get
(under assumption that $\delta_n\leq (1-\gamma)\frac{\bar g_r}{2}$)
\begin{align*}
&
\biggl|{\mathbb E}\Bigl\langle P_r^{\perp}S_r(E)P_rv,u\Bigr\rangle
I(\|E\|_{\infty}\leq \delta_n)\biggr| \leq
\\
&
\nonumber  
\sum_{k\geq 2}\frac{1}{2\pi}\oint_{\gamma_r} \biggl[\frac{1}{|\mu_r-\eta|^2}\sum_{s=2}^{k} \Bigl|{\mathbb E}\Bigl\langle (\tilde R_{\Sigma}(\eta)E)^{s-1}P_rE (R_{\Sigma}(\eta)E)^{k-s}P_rv,u\Bigr\rangle I(\|E\|_{\infty}\leq \delta_n)\Bigr| +
\\
&
\nonumber 
\frac{1}{|\mu_r-\eta|}\Bigl|{\mathbb E}\Bigl\langle(\tilde R_{\Sigma}(\eta)E)^k P_rv,u\Bigr\rangle I(\|E\|_{\infty}\leq \delta_n)\Bigr|\biggr]d\eta
\\
&
\nonumber 
\leq 
\sum_{k\geq 2} \frac{1}{2\pi} 2\pi \frac{\bar g_r}{2} \biggl(\frac{2}{\bar g_r}\biggr)^2 \sqrt{2}m_r\frac{\|\Sigma\|_{\infty}}{\sqrt{n}} k \biggl(\frac{2\delta_n}{\bar g_r}\biggr)^{k-1} \|u\|\|v\|
\\
&
\nonumber
=\sqrt{2}m_r\frac{\|\Sigma\|_{\infty}}{\sqrt{n}}\frac{2}{\bar g_r}
\sum_{k\geq 2}k \biggl(\frac{2\delta_n}{\bar g_r}\biggr)^{k-1}\|u\|\|v\|
\\
&
\nonumber
=\sqrt{2}m_r\frac{\|\Sigma\|_{\infty}}{\sqrt{n}}\frac{2}{\bar g_r}
\biggl(\biggl(1-\frac{2\delta_n}{\bar g_r}\biggr)^{-2}-1\biggr)\|u\|\|v\|\leq 
\frac{8\sqrt{2}}{\gamma^2} m_r\frac{\|\Sigma\|_{\infty}}{\sqrt{n}}\frac{\delta_n}{\bar g_r^2}\|u\|\|v\|.
\end{align*}
Therefore, 
\begin{equation}
\label{Prperp''}
\Bigl\|{\mathbb E}P_r^{\perp}S_r(E)P_r
I(\|E\|_{\infty}\leq \delta_n)\Bigr\|_{\infty}\leq 
\frac{8\sqrt{2}}{\gamma^2} m_r\frac{\|\Sigma\|_{\infty}}{\sqrt{n}}\frac{\delta_n}{\bar g_r^2}.
\end{equation}
Obviously, the same bound holds for 
$\Bigl\|{\mathbb E}P_r S_r(E)P_r^{\perp}
I(\|E\|_{\infty}\leq \delta_n)\Bigr\|_{\infty}.$ 
Moreover, similarly, it can be 
proved that 
\begin{equation}
\label{Prperp'''}
\Bigl\|{\mathbb E}P_r^{\perp}S_r(E)P_r^{\perp}
I(\|E\|_{\infty}\leq \delta_n)\Bigr\|_{\infty} \leq
c_{\gamma}m_r\frac{\|\Sigma\|_{\infty}}{\sqrt{n}}\frac{\delta_n}{\bar g_r^2}.
\end{equation}
with some constant $c_{\gamma}>0.$

To complete the proof, note that
\begin{align*}
&
\Bigl\|{\mathbb E}\Bigl(P_r^{\perp}S_r(E)P_r+P_r S_r(E)P_r^{\perp}+
P_r^{\perp}S_r(E)P_r^{\perp}\Bigr)I(\|E\|_{\infty}>\delta_n)\Bigr\|_{\infty}
\\
&
\nonumber
\leq 
{\mathbb E}\Bigl\|P_r^{\perp}S_r(E)P_r+P_r S_r(E)P_r^{\perp}+
P_r^{\perp}S_r(E)P_r^{\perp}\Bigr\|_{\infty}I(\|E\|_{\infty}>\delta_n)
\\
&
\nonumber
\leq {\mathbb E}\|S_r(E)\|_{\infty}I(\|E\|_{\infty}>\delta_n).
\end{align*}

Next, using the formula $\E\|E\|_{\infty}^2I(\|E\|_{\infty}>\delta_n) = 2\int_{\delta_n}^{\infty}t \mathbb P \left( \|E\|_{\infty}>t\right)dt$ and bound (\ref{remainder_A}) of Lemma \ref{lem-pert-spectral}, 
we get
\begin{align*}
&
{\mathbb E}\|S_r(E)\|_{\infty} I(\|E\|_{\infty}> \delta_n)
\leq 
\frac{28}{\bar g_r^2}\int_{\delta_n}^{\infty}t \mathbb P \left( \|E\|_{\infty}>t\right)dt
\\
&
\leq 
\frac{28}{\bar g_r^2}\int_{0}^{\infty} (\delta_n + u) \mathbb P \left( \|E\|_{\infty}-\E\|E\|_{\infty}>u\right)du
\\
&
\leq\frac{28}{\bar g_r^2}\left( \delta_n  \int_{0}^{\infty} \mathbb P \left( |\|E\|_{\infty}-\E\|E\|_{\infty}|>u\right)du  +\int_{0}^{\infty} u\mathbb P \left( |\|E\|_{\infty}-\E\|E\|_{\infty}|>u\right)du \right).
%
\end{align*}
Set $Z:=|\|E\|_{\infty}-\E\|E\|_{\infty}|$. Recall that, in view of Theorem \ref{th_operator} and condition (\ref{condition gamma_r-1}), we have $\mathbf{r}(\Sigma)\lesssim n$. Combining this fact with the bound (\ref{con_con}) of Theorem \ref{spectrum_sharper}, we get
\begin{align*}
\int_{0}^{\infty} \mathbb P \left( Z>u\right)du &= \int_{0}^{n} \mathbb P \left( Z>u\right)du  + \int_{n}^{\infty}\mathbb P \left( Z>u\right)du\\
&=\int_{0}^{n} \mathbb P \left( Z>C\|\Sigma\|_{\infty} \sqrt{\frac{t}{n}}\right)\frac{ \|\Sigma\|_{\infty}}{2\sqrt{tn}}dt  + \int_{n}^{\infty} \mathbb{P} \left( Z>C\|\Sigma\|_{\infty} \frac{t}{n}\right)\frac{ \|\Sigma\|_{\infty}}{n}dt\\
&\lesssim \|\Sigma\|_{\infty} \frac{1}{\sqrt{n}},
\end{align*}
By a similar reasoning, we get that
\begin{align*}
\int_{0}^{\infty} u  \mathbb  P  \left( Z>u\right)du \lesssim \|\Sigma\|_{\infty}^2 \frac{1}{n}.
\end{align*}
Combining the last three displays with Theorem \ref{th_operator}, we get 
\begin{align*}
&
{\mathbb E}\|S_r(E)\|_{\infty} I(\|E\|_{\infty}> \delta_n)
\lesssim
\frac{1}{\bar g_r^2} \|\Sigma\|_{\infty} \left(\delta_n\frac{1}{\sqrt{n}}  + \|\Sigma\|_{\infty}\frac{1}{n}\right).
\end{align*}
It remains to observe that the condition (\ref{condition gamma_r-1}) and Theorem \ref{th_operator} gives $\delta_n \lesssim \|\Sigma\|_{\infty} \sqrt{\frac{\mathbf{r}(\Sigma)}{n}}$. Consequently, we get that 
\begin{align}
\label{kone}
&
\Bigl\|{\mathbb E}\Bigl(P_r^{\perp}S_r(E)P_r+P_r S_r(E)P_r^{\perp}+
P_r^{\perp}S_r(E)P_r^{\perp}\Bigr)I(\|E\|_{\infty}>\delta_n)\Bigr\|_{\infty}
\lesssim \frac{\|\Sigma\|_{\infty}^2}{\bar g_r^2}\sqrt{\frac{{\bf r}(\Sigma)}{n}}\frac{1}{\sqrt{n}}.
\end{align}
Bound (\ref{T_r}) now follows from representation (\ref{SrPrE}), bounds (\ref{Prperp''}), (\ref{Prperp'''}) and (\ref{kone}).
\qed
\end{proof}

\section{Asymptotics of Bilinear Forms of Empirical Spectral Projectors}\label{Sec:Asymptotic}

In this section, we study the asymptotic behavior of the bilinear forms  
$$\Bigl\langle (\hat P_r -{\mathbb E}\hat P_r) u,v\Bigr\rangle, u,v\in {\mathbb H}$$
in the case when the sample size $n$ and the effective rank ${\bf r}(\Sigma)$ are both large. To describe this precisely, one has to deal with a sequence of 
problems in which the data is sampled from Gaussian distributions in ${\mathbb H}$
with mean zero and covariance $\Sigma=\Sigma^{(n)}.$ This leads to the following 
asymptotic framework. Let $X=X^{(n)}$ be a centered Gaussian random vector in ${\mathbb H}$ with covariance operator $\Sigma=\Sigma^{(n)}$ and let  
$X_1=X_1^{(n)},\dots, X_n=X_n^{(n)}$ be i.i.d. copies of $X^{(n)}.$ The sample covariance based on $(X_1^{(n)},\dots, X_n^{(n)})$ is 
denoted by $\hat \Sigma_n.$ 
Let $\sigma(\Sigma^{(n)})$ be the spectrum of $\Sigma^{(n)},$ $\mu_r^{(n)}, r\geq 1$ be distinct nonzero eigenvalues of $\Sigma^{(n)}$ arranged in decreasing order  and $P_r^{(n)}, r\geq 1$ be the corresponding spectral projectors.
As before, denote $\Delta_r^{(n)}:=\{j: \sigma_j(\Sigma^{(n)})=\mu_r^{(n)}\}$
and let $\hat P_r^{(n)}$ be the orthogonal projector on the direct sum 
of eigenspaces corresponding to the eigenvalues $\{\sigma_j(\hat \Sigma_n), j\in \Delta_r^{(n)}\}.$

The next assumption means that, for large enough $n,$
there exists a unique eigenvalue $\mu^{(n)}$ of $\Sigma^{(n)}$ isolated  
inside a fixed interval from the rest of the spectrum of $\Sigma^{(n)}.$

\begin{assumption}
\label{separate_spectrum}
There exists an interval $(\alpha,\beta)\subset {\mathbb R}_{+}$ and a 
number $\delta>0$ 
such that, for all large enough $n,$ the set $\sigma(\Sigma^{(n)})\cap (\alpha,\beta)$ consists of a single eigenvalue $\mu^{(n)}=\mu_{r_n}^{(n)}$ of $\Sigma^{(n)}$ and    
$$
\sigma(\Sigma^{(n)})\setminus \{\mu^{(n)}\} \subset {\mathbb R}_{+}
\setminus (\alpha-\delta, \beta+\delta).
$$
\end{assumption}

Denote by $P^{(n)}$ the spectral projector corresponding  
to the eigenvalue $\mu^{(n)}$ and define the following sequence of operators: 
$$
C^{(n)}:=\sum_{\mu_s^{(n)}\neq \mu^{(n)}}\frac{1}{\mu^{(n)}-\mu_s^{(n)}}P_s^{(n)}.
$$

Consider the spectral measures associated with 
the covariance operators $\Sigma^{(n)}:$ 
$$
\Lambda^{(n)}_{u,v}(A):=\sum_{r=1}^{\infty}
\Bigl\langle P_r^{(n)}u,v\Bigr\rangle I_A(\mu_r^{(n)}), u,v\in {\mathbb H}, A\in 
{\mathcal B}({\mathbb R}_+),
$$
where ${\mathcal B}({\mathbb R}_+)$ denotes the Borel $\sigma$-algebra in ${\mathbb R}_+.$

\begin{assumption}
\label{weak_spectral_measure}
For all $u,v\in {\mathbb H},$ the sequence of measures $\Lambda^{(n)}_{u,v}$ converges weakly to a measure $\Lambda_{u,v}$
in ${\mathbb R}_{+}.$
Also assume that there exists $u\in {\mathbb H}$ such that 
$\Lambda_{u,u}([\alpha,\beta])>0.$
\end{assumption}
 
It turns out that the following assumption, which is somewhat easier to understand,
implies Assumption \ref{weak_spectral_measure} and even its stronger version.

\begin{assumption}
\label{weak_spectral_measure''}
Suppose the sequence of covariance operators $\Sigma^{(n)}$ with 
$\sup_{n\geq 1}\|\Sigma^{(n)}\|_{\infty}<+\infty$ converges strongly to a bounded 
symmetric nonnegatively definite operator $\Sigma:{\mathbb H}\mapsto {\mathbb H}$ (that is, $\Sigma^{(n)}u\to \Sigma u$ as $n\to\infty$ for all $u\in {\mathbb H}$). Let $E(\cdot)$ be the decomposition
of identity associated with $\Sigma.$\footnote{This means that $E(\cdot)$ is a projector 
valued measure on Borel subsets of ${\mathbb R}_{+},$ such that $E(\Delta)E(\Delta')=E(\Delta\cap \Delta'),$ $E({\mathbb R}_+)=I$ and $\Sigma=\int_{{\mathbb R}_{+}}\lambda E(d\lambda).$} 
Suppose also that there exists $u\in {\mathbb H}$ such that 
$\Bigl\langle E([\alpha,\beta])u,u\Bigr\rangle>0.$
\end{assumption}

\begin{proposition}
\label{prop_dec_id}
Assumption \ref{weak_spectral_measure''} implies Assumption \ref{weak_spectral_measure}. Moreover, it implies that, for all $u,v\in {\mathbb H}$ and for all sequences $u_n\to u,$ $v_n\to v$ as $n\to \infty,$ the sequence of 
measures $\Lambda^{(n)}_{u_n,v_n}$ converges weakly to $\Lambda_{u,v}.$
\end{proposition}

\begin{proof}
Under Assumption \ref{weak_spectral_measure''}, define 
$$
\Lambda_{u,v}(\Delta)=\Bigl\langle E(\Delta)u,v\Bigr\rangle, \Delta \in {\mathcal B}({\mathbb R}_+), u,v\in {\mathbb H}. 
$$
Let $E^{(n)}(\cdot)$ be the decomposition of identity associated with $\Sigma^{(n)}.$
Then $\Lambda^{(n)}_{u,v}(\cdot)=\Bigl\langle E^{(n)}(\cdot)u,v\Bigr\rangle.$
It is well known (see, e.g., \cite{Riesz}, Ch. IX, Section 134) that the uniform boundedness of $\|\Sigma^{(n)}\|_{\infty}$ and strong convergence of operators $\Sigma^{(n)}$ to $\Sigma$ implies strong convergence of $E^{(n)}([0,\lambda])$ to $E([0,\lambda])$ for 
all $\lambda$ that do not belong to the point spectrum of $\Sigma,$ 
which easily implies the weak convergence of measures $\Lambda^{(n)}_{u_n,v_n}$ to $\Lambda_{u,v}.$  
\qed
\end{proof}

We will also need the following simple proposition (its
proof is elementary).

\begin{proposition}
\label{prop_dec_1}
Suppose assumptions \ref{separate_spectrum} and \ref{weak_spectral_measure''} hold.
Suppose also that $\mu^{(n)}$ is an eigenvalue of multiplicity $1.$ Then, the 
corresponding spectral projector $P^{(n)}=\theta^{(n)}\otimes \theta^{(n)},$
where $\theta^{(n)}$ is the eigenvector corresponding to $\mu^{(n)}$ and, for 
some $\theta\in {\mathbb H},$  
$\theta^{(n)}\to \theta$ as $n\to\infty.$  
\end{proposition}

As a typical example where Assumption \ref{weak_spectral_measure''} holds, 
consider the case of $\Sigma^{(n)}=P_{L_n} \Sigma P_{L_n}$ for a sequence 
of subspaces $L_n\subset {\mathbb H}$ with ${\rm dim}(L_n)\to \infty$
and $\bigcup_{n\geq 1}L_n$ being dense in ${\mathbb H}$ (see also the discussion 
of general spiked covariance models in Section \ref{Sec:Intro}).

Denote
$$
\Gamma_1(u,v):=\int_{\alpha}^{\beta} \lambda \Lambda_{u,v}(d\lambda),\ \ 
\Gamma_2(u,v):= 
\int_{{\mathbb R}_+\setminus [\alpha,\beta]} \frac{\lambda}{(\mu-\lambda)^2} \Lambda_{u,v}(d\lambda)
$$
and 
\begin{align}
\nonumber
&
\Gamma (u,v;u',v'):=
\\
\nonumber
&\Gamma_1(v,v')\Gamma_2(u,u')+
\Gamma_1(v,u')\Gamma_2(u,v')+
\Gamma_1(u,u')\Gamma_2(v,v')+
\Gamma_1(u,v')\Gamma_2(v,u').
\end{align}

\begin{theorem}
\label{asymptotic_normality_1}
Suppose that 
\begin{equation}
\label{bounded_norm}
\sup_{n\geq 1}\|\Sigma^{(n)}\|_{\infty}<\infty
\end{equation}
and 
\begin{equation}
\label{assume_rank}
{\bf r}(\Sigma^{(n)})=o(n)\ {\rm as}\ n\to\infty.
\end{equation}
Also, suppose that assumptions \ref{separate_spectrum} and \ref{weak_spectral_measure}
hold. Let $\hat P^{(n)}:=\hat P_{r_n}^{(n)}.$
Then, 
the finite dimensional distributions of stochastic processes
$$
n^{1/2}\Bigl\langle (\hat P^{(n)}-{\mathbb E}\hat P^{(n)})u,v\Bigr\rangle,\ u,v\in {\mathbb H} 
$$
converge weakly as $n\to \infty$ 
to the finite dimensional distributions of the centered Gaussian process $Y(u,v), u,v\in {\mathbb H}$
with covariance function $\Gamma.$ 

If, in addition, Assumption \ref{weak_spectral_measure''}
holds, then, for all $\varphi_n, \psi_n:{\mathbb H}\mapsto {\mathbb H}$
such that $\varphi_n(u)\to u, \psi_n(u)\to u$ as $n\to\infty$ for all $u\in {\mathbb H},$ 
the finite dimensional distributions of stochastic processes
$$
n^{1/2}\Bigl\langle (\hat P^{(n)}-{\mathbb E}\hat P^{(n)})\varphi_n(u),\psi_n(v)\Bigr\rangle,\ u,v\in {\mathbb H} 
$$
converge weakly as $n\to \infty$ to the same limit.
\end{theorem}

\begin{proof} 
We prove only the first claim. The modifications needed to establish the second
claim are rather obvious. 
The proof is based on the following representation of $\hat P^{(n)}-P^{(n)}:$
\begin{equation}
\label{P^{(n)}}
\hat P^{(n)}-{\mathbb E}\hat P^{(n)} = L^{(n)}(E^{(n)})+ R^{(n)},
\end{equation}
where 
$$
L^{(n)}(E^{(n)})= P^{(n)}E^{(n)}C^{(n)}+C^{(n)}E^{(n)}P^{(n)}, 
\ \ E^{(n)}:=\hat \Sigma_n-\Sigma^{(n)}
$$
and where the remainder $R^{(n)}$ will be controlled using Theorem \ref{technical_1}. 

In addition to this, to show the asymptotic normality of 
$\Bigl\langle L^{(n)}(E^{(n)})u,v\Bigr\rangle,$
we need a couple of lemmas based on assumptions \ref{separate_spectrum} and \ref{weak_spectral_measure}.

\begin{lemma}
\label{cov_lim}
Under the assumptions \ref{separate_spectrum} and \ref{weak_spectral_measure},
the following statements hold.\\
(i) There exists $\mu\in [\alpha,\beta]$ such that 
$$
\mu^{(n)}\to \mu\ {\rm as}\ n\to\infty. 
$$
\\
(ii) For all $u,v\in {\mathbb H},$ 
$$
\Bigl\langle P^{(n)}u,v\Bigr\rangle \to \Lambda_{u,v}([\alpha,\beta])\ {\rm as}\ n\to\infty.
$$
\\
(iii) For all $u,v\in {\mathbb H},$
$$
\Bigl\langle P^{(n)}\Sigma^{(n)}P^{(n)}u,v\Bigr\rangle \to 
\Gamma_1(u,v)\ {\rm as}\ n\to\infty.
$$
\\
(iv) For all $u,v\in {\mathbb H},$
$$
\Bigl\langle C^{(n)}\Sigma^{(n)}C^{(n)}u,v\Bigr\rangle \to 
\Gamma_2(u,v)\ {\rm as}\ n\to\infty.
$$
\end{lemma}

\begin{proof}
We start with proving (ii). 
In view of Assumption \ref{separate_spectrum}, for all $\delta'<\delta,$
$$
\Lambda^{(n)}_{u,v}((\alpha-\delta',\beta+\delta'))=
\Lambda^{(n)}_{u,v}(\{\mu^{(n)}\})=\langle P^{(n)}u,v\rangle.
$$
We can choose $\delta'$ such that $\alpha-\delta'$ and $\beta+\delta'$
are not atoms of $\Lambda_{u,v}.$ Therefore, 
by Assumption \ref{weak_spectral_measure},
$$
\Bigl\langle P^{(n)}u,v\Bigr\rangle =\Lambda^{(n)}_{u,v}((\alpha-\delta',\beta+\delta'))\to 
\Lambda_{u,v}((\alpha-\delta',\beta+\delta'))\ {\rm as}\ n\to\infty
$$
for all such $\delta'.$ Note that the limit does not depend 
on $\delta'.$ It is enough now to let $\delta'\to 0$
to get (ii). 

To prove (iii), note that, for the same $\delta'$ as in the previous step,
$$
\Bigl\langle P^{(n)}\Sigma^{(n)}P^{(n)}u,v\Bigr\rangle=
\int_{\alpha-\delta'}^{\beta+\delta'}\lambda \Lambda^{(n)}_{u,v}(d\lambda)
\to 
\int_{\alpha-\delta'}^{\beta+\delta'}\lambda \Lambda_{u,v}(d\lambda),
$$
and, again, it is enough to let $\delta'\to 0.$

To prove (i), take $v=u\in {\mathbb H}$ such that $\Lambda_{u,u}([\alpha,\beta])>0.$ By (iii), we have 
$$
\mu^{(n)}\Bigl\langle P^{(n)}u,u\Bigr\rangle=
\Bigl\langle P^{(n)}\Sigma^{(n)}P^{(n)}u,u\Bigr\rangle
\to \int_{\alpha}^{\beta} \lambda \Lambda_{u,v}(d\lambda)
$$
and, by (ii), 
$$
\langle P^{(n)}u,u\rangle \to \Lambda_{u,u}([\alpha,\beta])>0.
$$
This implies that 
$$
\mu^{(n)}\to \mu :=\frac{\int_{\alpha}^{\beta}\lambda \Lambda_{u,u}(d\lambda)}{\Lambda_{u,u}([\alpha,\beta])}
$$
that clearly belongs to $[\alpha,\beta]$ (and does not depend 
on the choice of $u$). 

Finally, we prove (iv). 
To this end, note that, for all $\delta'<\delta,$
$$
\Bigl\langle C^{(n)}\Sigma^{(n)}C^{(n)}u,v\Bigr\rangle 
=\int_{{\mathbb R}_+\setminus (\alpha-\delta',\beta+\delta')}
\frac{\lambda}{(\mu^{(n)}-\lambda)^2}\Lambda^{(n)}_{u,v}(d\lambda). 
$$
Due to bilinearity, it will be enough to consider the case when $v=u.$
Let $\delta'<\delta$ and suppose that $\alpha-\delta',\beta+\delta'$
are not atoms of $\Lambda_{u,u}.$ Since $\mu^{(n)}\to \mu$ and 
Assumption \ref{separate_spectrum} holds, 
$$
\frac{\lambda}{(\mu^{(n)}-\lambda)^2}\to \frac{\lambda}{(\mu-\lambda)^2}
\ {\rm as}\ n\to\infty
$$
uniformly in ${\mathbb R}_+\setminus (\alpha-\delta',\beta+\delta').$
Due to the weak convergence of $\Lambda^{(n)}_{u,u}$ to $\Lambda_{u,u},$
it is easy to show that  
$$
\Bigl\langle C^{(n)}\Sigma^{(n)}C^{(n)}u,u\Bigr\rangle=
\int_{{\mathbb R}_+\setminus (\alpha-\delta',\beta+\delta')}
\frac{\lambda}{(\mu^{(n)}-\lambda)^2}\Lambda^{(n)}_{u,u}(d\lambda)
$$
$$
\to 
\int_{{\mathbb R}_+\setminus (\alpha-\delta',\beta+\delta')}
\frac{\lambda}{(\mu-\lambda)^2}\Lambda_{u,u}(d\lambda),
$$
and it remains to let $\delta'\to 0.$
\qed
\end{proof}

Observe that  
\begin{equation}
\label{sumiid}
n^{1/2}\Bigl\langle L^{(n)}(E^{(n)})u,v\Bigr\rangle = n^{-1/2}\sum_{j=1}^n\Bigl(\xi_j^{(n)}(u,v)+\xi_j^{(n)}(v,u)\Bigr),
\end{equation}
where 
$
\xi_j^{(n)}(u,v):=
\Bigl\langle X_j^{(n)},P^{(n)} v\Bigr\rangle \Bigl\langle X_j^{(n)}, C^{(n)}u\Bigr\rangle
$
are independent copies of random variable 
$
\xi^{(n)}(u,v):=
\Bigl\langle X^{(n)},P^{(n)} v\Bigr\rangle 
\Bigl\langle X^{(n)} C^{(n)}u\Bigr\rangle.
$
Recall also that Gaussian random variables $\Bigl\langle X^{(n)},P^{(n)} v\Bigr\rangle,$ $\Bigl\langle X^{(n)}, C^{(n)}u\Bigr\rangle$ are uncorrelated 
and, hence, independent. Therefore,
$\xi^{(n)}(u,v)$ is mean zero and, by Lemma \ref{cov_lim}, for 
all $u,v,u',v'\in {\mathbb H},$ 
$$
{\mathbb E}\xi^{(n)}(u,v)\xi^{(n)}(u',v')= 
\langle P^{(n)}\Sigma^{(n)}P^{n}v,v'\rangle 
\langle C^{(n)}\Sigma^{(n)}C^{n}u,u'\rangle 
\to 
\bar \Gamma(u,v;u',v'):=\Gamma_1(v,v')\Gamma_2(u,u'),
$$
which implies 
\begin{align}
\nonumber
&{\mathbb E}(\xi^{(n)}(u,v)+\xi^{(n)}(v,u))(\xi^{(n)}(u',v')+\xi^{(n)}(v',u'))\to 
\Gamma (u,v;u'v').
\end{align}

\begin{lemma}
\label{linear_clt}
Under the assumptions \ref{separate_spectrum} and \ref{weak_spectral_measure}, 
the sequence of finite dimensional distributions of 
$$
n^{1/2}\Bigl\langle L^{(n)}(E^{(n)})u,v\Bigr\rangle, u,v\in {\mathbb H}
$$
converges weakly as $n\to\infty$ to the finite dimensional 
distributions of the centered Gaussian process $Y(u,v), u,v\in {\mathbb H}$
with covariance function $\Gamma.$
\end{lemma}

\begin{proof}
In view of (\ref{sumiid}),
it is enough to show the convergence of finite dimensional distributions 
of the process 
$
n^{-1/2}\sum_{j=1}^n \xi_j^{(n)}(u,v), u,v\in {\mathbb H}
$
to the finite dimensional distributions of the centered Gaussian 
process $\bar Y(u,v), u,v\in {\mathbb H}$ with covariance function 
$\bar \Gamma.$ 
To this end, one has to check the Lindeberg 
condition, which reduces to 
$$
\frac{{\mathbb E}|\xi^{(n)}(u,v)|^2 I\Bigl(|\xi^{(n)}(u,v)|\geq \tau \sqrt{n} {\mathbb E}^{1/2}|\xi^{(n)}(u,v)|^2\Bigr)}{{\mathbb E}|\xi^{(n)}(u,v)|^2}
\to 0\ {\rm as}\ n\to\infty
$$ 
for all $\tau>0.$ 
Note that
$$
\frac{{\mathbb E}|\xi^{(n)}(u,v)|^2 I\Bigl(|\xi^{(n)}(u,v)|\geq \tau \sqrt{n} {\mathbb E}^{1/2}|\xi^{(n)}(u,v)|^2\Bigr)}{{\mathbb E}|\xi^{(n)}(u,v)|^2}
\leq 
\frac{{\mathbb E}|\xi^{(n)}(u,v)|^4}{\tau^2 n \Bigl({\mathbb E}|\xi^{(n)}(u,v)|^2\Bigr)^2}.
$$
Since 
$$
{\mathbb E}|\xi^{(n)}(u,v)|^2= 
\Bigl\langle P^{(n)}\Sigma^{(n)}P^{(n)}v,v\Bigr\rangle
\Bigl\langle C^{(n)}\Sigma^{(n)}C^{(n)}u,u\Bigr\rangle
$$
and 
$$
{\mathbb E}|\xi^{(n)}(u,v)|^4=
{\mathbb E}\Bigl\langle X^{(n)},P^{(n)} v\Bigr\rangle^4 
{\mathbb E}\Bigl\langle X^{(n)}, C^{(n)}u\Bigr\rangle^4
=
9\Bigl\langle P^{(n)}\Sigma^{(n)}P^{(n)}v,v\Bigr\rangle^2
\Bigl\langle C^{(n)}\Sigma^{(n)}C^{(n)}u,u\Bigr\rangle^2
$$
(where we used the fact that, for a centered normal random variable $g,$
${\mathbb E}g^4=3({\mathbb E}g^2)^2$), we get 
$$
\limsup_{n\to\infty}\frac{{\mathbb E}|\xi^{(n)}(u,v)|^4}{\tau^2 n\Bigl({\mathbb E}|\xi^{(n)}(u,v)|^2\Bigr)^2} 
= \frac{9 \Bigl\langle P^{(n)}\Sigma^{(n)}P^{(n)}v,v\Bigr\rangle^2
\Bigl\langle C^{(n)}\Sigma^{(n)}C^{(n)}u,u\Bigr\rangle^2}
{\Bigl\langle P^{(n)}\Sigma^{(n)}P^{(n)}v,v\Bigr\rangle^2
\Bigl\langle C^{(n)}\Sigma^{(n)}C^{(n)}u,u\Bigr\rangle^2}
\lim_{n\to\infty}\frac{1}{\tau^2 n}=0,
$$
and the result follows.
\qed
\end{proof}

To complete the proof of Theorem \ref{asymptotic_normality_1}, it is enough to 
use representation (\ref{P^{(n)}})
and bound (\ref{remaind}) of Theorem \ref{technical_1}. 
Since ${\bf r}(\Sigma^{(n)})=o(n),$ it follows from bound (\ref{remaind}) that 
$$
\langle R^{(n)}u,v\rangle = o_{\mathbb P}(n^{-1/2}),
$$
and the result follows from Lemma \ref{linear_clt}. \qed
\end{proof}

\begin{remark}
Under the assumption 
\begin{equation}
\label{assume_rank_A}
{\bf r}(\Sigma^{(n)})=o(n^{1/2})\ {\rm as}\ n\to\infty,
\end{equation}
the finite dimensional distributions of stochastic processes
$$
n^{1/2}\Bigl\langle (\hat P^{(n)}-P^{(n)})u,v\Bigr\rangle,\ u,v\in {\mathbb H} 
$$
converge weakly as $n\to \infty$ 
to the finite dimensional distributions of $Y.$
Indeed, by Theorem \ref{bd_bias} and bound (\ref{W_r}),
$$
\|{\mathbb E}\hat P^{(n)}-P^{(n)}\|_{\infty}=O\biggl(\frac{{\bf r}(\Sigma^{(n)})}{n}\biggr)=o(n^{-1/2}),
$$
and the claim follows from Theorem \ref{asymptotic_normality_1}.
\end{remark}

\section{Asymptotics and Concentration Bounds for Linear Forms of Empirical Eigenvectors Corresponding to a Simple Eigenvalue}\label{Sec:VarSel}

We will discuss special versions of some of the results of the previous sections 
in the case of spectral projectors corresonding to an isolated simple eigenvalue.
In this case, it becomes natural to state the results in terms of eigenvectors rather than spectral projectors. 

Suppose $\mu_r$ is a simple eigenvalue of $\Sigma,$ that is, $\mu_r$ is of multiplicity $m_r=1$ so that 
the spectral projector $P_r$ is of rank $1:$ $P_r=\theta_r\otimes \theta_r,$ where $\theta_r$ is a unit eigenvector corresponding to $\mu_r.$ Under the assumptions of Theorem \ref{bd_bias},
$$
{\mathbb E}\hat P_r=P_r+P_rW_rP_r + T_r,
$$ 
where the remainder $T_r$ satisfies bound (\ref{T_r}):
\begin{equation}
\label{T_r'''}
\|T_r\|_{\infty}\leq D_{\gamma}\frac{\|\Sigma\|_{\infty}^2}{\bar g_r^2}
\sqrt{\frac{{\bf r}(\Sigma)}{n}}\frac{1}{\sqrt{n}}.
\end{equation}
Note that 
$$
\langle P_r W_r P_r u,v\rangle = \langle P_r W_r \theta_r,v\rangle \langle \theta_r,u\rangle= \langle W_r\theta_r,\theta_r\rangle \langle \theta_r,u\rangle 
\langle \theta_r,v\rangle. 
$$
Therefore, $P_r W_r P_r = b_r P_r$
and 
\begin{equation}
\label{Ehat}
{\mathbb E}\hat P_r = (1+b_r)P_r + T_r,
\end{equation}
where $b_r:=\langle W_r \theta_r,\theta_r\rangle$ is a real number
characterizing the bias of $\hat P_r.$ 
It follows from (\ref{Ehat}) that 
$$
{\mathbb E}\langle \hat P_r \theta_r,\theta_r \rangle = 1+b_r+ 
\langle T_{r}\theta_r,\theta_r\rangle.
$$
Since $0\leq \langle \hat P_r \theta_r,\theta_r \rangle\leq 1,$ this implies 
that 
$$
-1-\|T_r\|_{\infty}\leq b_r \leq \|T_r\|_{\infty}.
$$
Under natural assumptions, $\|T_r\|_{\infty}=O(n^{-1/2}),$ so, we have that 
$b_r$ is between $-1+O(n^{-1/2})$ and $O(n^{-1/2}).$ In what follows, we will often 
assume that $b_r$ is bounded away from $-1$ which would ensure that the bias is 
not too large. In fact, it follows from bound (\ref{W_r}) that, under the assumption that ${\bf r}(\Sigma)\lesssim n,$ 
\begin{equation}
\label{bdbr}
|b_r|\lesssim \frac{\|\Sigma\|_{\infty}^2}{\bar g_r^2}\frac{{\bf r}(\Sigma)}{n},
\end{equation}
so, $b_r$ is small provided that $\frac{\|\Sigma\|_{\infty}}{\bar g_r}$ remains 
bounded and ${\bf r}(\Sigma)=o(n).$

In what follows, assume that $\hat P_r =\hat \theta_r\otimes \hat \theta_r$ 
and the sign of $\hat \theta_r$ is chosen in such a way 
that $\langle \hat \theta_r,\theta_r \rangle\geq 0.$
Since the eigenvectors $\hat \theta_r,\theta_r$ are defined only up to their signs, there is no loss of generality in such an assumption.

\begin{theorem}
\label{th:hattheta}
Let $t\geq 1$ and $\gamma\in (0,1/2).$
There exists a constant $C_{\gamma}>0$ such that, if 
$$
{\mathbb E}\|\hat \Sigma-\Sigma\|_{\infty}\leq \frac{(1-2\gamma)\bar g_r}{2},\ \  
1+b_r\geq 2\gamma.
$$ 
and
$$
C_{\gamma}\|\Sigma\|_{\infty}\biggl(\sqrt{\frac{t}{n}}\bigvee \frac{t}{n} \biggr)
\leq \frac{\gamma \bar g_r}{2},
$$
then with probability at least 
$1-e^{-t}$
$$
\Bigl|\Bigl\langle \hat \theta_r-\sqrt{1+b_r}\theta_r,u\Bigr\rangle\Bigr|
\leq C_{\gamma} \frac{\|\Sigma\|_{\infty}}{\bar g_r}
\sqrt{\frac{t}{n}}\|u\|.
$$  
\end{theorem}

\begin{rema}
It is easy to see that the assumptions of the theorem hold provided that 
$\frac{\|\Sigma\|_{\infty}}{\bar g_r}$ is bounded by a constant and $n$
is sufficiently large so that $n\gtrsim ({\bf r}(\Sigma)\vee t).$ 
\end{rema}

\begin{proof}
We need the following lemma 
that provides a representation of the linear functional 
$\langle \hat \theta_r-\theta_r, u\rangle$ in terms of bilinear form of operator $\hat P_r-P_r.$

\begin{lemma}\label{lem:single_eig}
For all $u\in {\mathbb H},$ 
\begin{equation}
\label{simeig}
\langle \hat \theta_r-\theta_r,u\rangle
=\frac{\langle (\hat P_r-P_r)\theta_r,u\rangle-
\Bigl(\sqrt{1+\langle (\hat P_r-P_r)\theta_r,\theta_r\rangle}-1\Bigr)\langle \theta_r, u\rangle}{\sqrt{1+\langle (\hat P_r-P_r)\theta_r,\theta_r\rangle}} 
\end{equation}
\end{lemma}

\begin{proof}
The following representation is obvious
$$
(\hat P_r-P_r)\theta_r = \hat \theta_r-\theta_r +
\langle \hat \theta_r-\theta_r, \theta_r \rangle \theta_r
+  
\langle \hat \theta_r-\theta_r, \theta_r \rangle (\hat \theta_r-\theta_r)
$$
and it implies that 
\begin{equation}
\label{hattheta_dot_u}
\langle \hat \theta_r-\theta_r,u\rangle= 
\frac{\langle (\hat P_r-P_r)\theta_r,u\rangle-\langle \hat \theta_r-\theta_r, \theta_r \rangle \langle \theta_r, u\rangle}{1+\langle \hat \theta_r-\theta_r, \theta_r\rangle}.
\end{equation}
For $u=\theta_r,$ it yields 
$$
\langle \hat \theta_r-\theta_r, \theta_r\rangle^2 + 2\langle \hat \theta_r-\theta_r, \theta_r \rangle
= \langle (\hat P_r-P_r)\theta_r,\theta_r\rangle
$$
and, since $\langle \hat \theta_r, \theta_r\rangle\geq  0,$ we easily get that
$$
\langle \hat \theta_r, \theta_r\rangle= \sqrt{1+\langle (\hat P_r-P_r)\theta_r,\theta_r\rangle}.
$$
Substituting this into (\ref{hattheta_dot_u}) gives the result. 
\qed
\end{proof}

Denote 
$$
\rho_r(u):=\langle (\hat P_r-(1+b_r)P_r)\theta_r,u\rangle.
$$
We can rewrite (\ref{simeig}) as follows:
\begin{align*}
&
\nonumber
\langle \hat \theta_r-\theta_r,u\rangle =
\frac{b_r\langle \theta_r,u\rangle + \rho_r(u)-\Bigl(\sqrt{1+b_r+\rho_r(\theta_r)}-1\Bigr)\langle \theta_r,u\rangle}
{\sqrt{1+b_r+\rho_r(\theta_r)}}
\\
&
\nonumber
=\biggl(\frac{1+b_r}{\sqrt{1+b_r+\rho_r(\theta_r)}}-1\biggr)\langle \theta_r,u\rangle
+\frac{\rho_r(u)}{\sqrt{1+b_r+\rho_r(\theta_r)}}
\\
&
\nonumber
=\Bigl(\sqrt{1+b_r}-1\Bigr)\langle\theta_r,u\rangle + 
\biggl(\frac{1+b_r}{\sqrt{1+b_r+\rho_r(\theta_r)}}-\sqrt{1+b_r}\biggr)\langle \theta_r,u\rangle + \frac{\rho_r(u)}{\sqrt{1+b_r+\rho_r(\theta_r)}}
\\
&
\nonumber
=\Bigl(\sqrt{1+b_r}-1\Bigr)\langle\theta_r,u\rangle 
+ 
\frac{\rho_r(u)}{\sqrt{1+b_r+\rho_r(\theta_r)}}
\\
&
- 
\frac{\sqrt{1+b_r}}{\sqrt{1+b_r+\rho_r(\theta_r)}\Bigl(\sqrt{1+b_r+\rho_r(\theta_r)}+\sqrt{1+b_r}\Bigr)}
\rho_r(\theta_r)\langle \theta_r,u\rangle,
\end{align*}
which implies 
\begin{align}
\label{glav}
&
\nonumber
\Bigl\langle \hat \theta_r-\sqrt{1+b_r}\theta_r,u\Bigr\rangle =
\frac{\rho_r(u)}{\sqrt{1+b_r+\rho_r(\theta_r)}}
\\
&
- 
\frac{\sqrt{1+b_r}}{\sqrt{1+b_r+\rho_r(\theta_r)}\Bigl(\sqrt{1+b_r+\rho_r(\theta_r)}+\sqrt{1+b_r}\Bigr)}
\rho_r(\theta_r)\langle \theta_r,u\rangle.
\end{align}

%

The next bound on $\rho_r(u)$ follows from 
Corollary \ref{concent_bilin} and from Theorem \ref{bd_bias},
and it holds with probability at least $1-e^{-t}:$ 
\begin{align*}
&
\rho_r(u)
\leq D \frac{\|\Sigma\|_{\infty}}{\bar g_r}
\sqrt{\frac{t}{n}}\|u\|
+
D_\gamma \frac{\|\Sigma\|_{\infty}^2}{\bar g_r^2}
\biggl(\sqrt{\frac{{\bf r}(\Sigma)}{n}}\bigvee \sqrt{\frac{t}{n}}\bigvee \frac{t}{n}
\biggr)
\sqrt{\frac{t}{n}}\|u\|.
\end{align*}
The second term in the right hand side of the above bound can be dropped
since, for some constant $C_1>0,$ 
$$
{C_1}^{-1}\|\Sigma\|_{\infty}\biggl(\sqrt{\frac{{\bf r}(\Sigma)}{n}}\bigvee \sqrt{\frac{t}{n}}\bigvee\frac{t}{n} \biggr)<\frac{\bar g_r}{2},
$$ 
which implies the bound
\begin{equation}
\label{rhobd}
|\rho_r(u)|\leq (D+C_1D_{\gamma})
\frac{\|\Sigma\|_{\infty}}{\bar g_r}
\sqrt{\frac{t}{n}}\|u\|.
\end{equation}

Assume that $C_{\gamma}\geq D+C_1D_{\gamma}.$
Taking into account that, under the conditions of the theorem on an event of probability at least $1-e^{-t},$ 
$$
|\rho_r(\theta_r)|\leq 
(D+C_1D_{\gamma})\frac{\|\Sigma\|_{\infty}}{\bar g_r}
\sqrt{\frac{t}{n}}
\leq 
C_{\gamma}\frac{\|\Sigma\|_{\infty}}{\bar g_r}
\sqrt{\frac{t}{n}}
\leq \gamma/2,
$$
we get that 
$$
1+b_r + \rho_r(\theta_r)\geq \gamma.
$$
In view of bound (\ref{bdbr}) and Theorem \ref{th_operator}, it is also easy to see that, for some constant $C_1>0,$ 
\begin{equation}
\label{brbd}
|b_r|\leq C_1 \biggl(\frac{{\mathbb E}\|\hat \Sigma_n-\Sigma\|_{\infty}}{\bar g_r}\biggr)^2 \leq C_1 \biggl(\frac{1-2\gamma}{2}\biggr)^2 \leq C_1.
\end{equation}
Therefore, it follows from (\ref{glav}) and (\ref{rhobd}) that, for some $D_{\gamma'}>0,$ with probability at least $1-2e^{-t}$
$$
\Bigl|\Bigl\langle \hat \theta_r-\sqrt{1+b_r}\theta_r,u\Bigr\rangle\Bigr|
\leq D_{\gamma}' \frac{\|\Sigma\|_{\infty}}{\bar g_r}
\sqrt{\frac{t}{n}}\|u\|.
$$  
To complete the proof,
it is enough to take $C_{\gamma}= \max(C, D+C_1D_{\gamma}, D_{\gamma}')$
and also to adjust the constant properly (to write the probability bound as 
$1-e^{-t}$). 
\qed
\end{proof}

\begin{remark}
\label{rema_rhor}
In view of Remark \ref{rema_Cor1} of Section \ref{Sec:Representation}, the bound on $\rho_r(\theta_r)$ that appeared 
in the above proof could be improved as follows: with probability at least $1-e^{-t},$
$$
|\rho_r(\theta_r)|\leq 
D_\gamma \frac{\|\Sigma\|_{\infty}^2}{\bar g_r^2}
\biggl(\sqrt{\frac{{\bf r}(\Sigma)}{n}}\bigvee \sqrt{\frac{t}{n}} \bigvee \frac{t}{n}
\biggr)
\sqrt{\frac{t}{n}}.
$$ 
This implies that with the same probability
\begin{align}
\label{betterrho}
&
\Bigl|\Bigl\langle \hat \theta_r-\sqrt{1+b_r}\theta_r,\theta_r\Bigr\rangle\Bigr|
\leq 
C_\gamma \frac{\|\Sigma\|_{\infty}^2}{\bar g_r^2}
\biggl(\sqrt{\frac{{\bf r}(\Sigma)}{n}}\bigvee \sqrt{\frac{t}{n}}\bigvee \frac{t}{n}
\biggr)
\sqrt{\frac{t}{n}}.
\end{align}
\end{remark}

Based on Theorem \ref{th:hattheta}, it is easy to develop a simple $\sqrt{n}$-consistent estimator of the bias parameter $b_r$ and suggest 
an approach to bias reduction in the problem of estimation of linear 
functionals of eigenvectors of $\Sigma.$ Suppose, for simplicity, that
the sample size is an even number $2n$ and divide the sample $(X_1,\dots, X_{2n})$
into two subsamples of size $n$ each (the first $n$ observations and the rest).
Let $\hat \Sigma_n$ be the sample covariance based on the first subsample 
and $\hat \Sigma_n'$ be the sample covariance based on the second subsample.
For a simple eigenvalue $\mu_r$ with an eigenvector $\theta_r,$ denote  
by $\hat \theta_r$ the corresponding eigenvector of $\hat \Sigma_n$ and by $\hat \theta_r'$ the corresponding eigenvector of $\hat \Sigma_n'.$ Assume that their 
signs are chosen in such a way that $\langle\hat \theta_r, \hat \theta_r'\rangle\geq 0.$ Define 
$$
\hat b_r := \langle \hat \theta_r, \hat \theta_r'\rangle-1
$$
and 
$$
\tilde \theta_r := \frac{\hat \theta_r}{\sqrt{1+\hat b_r}}.
$$

\begin{proposition}\label{prop:tildetheta}
Under the assumptions and notations of Theorem \ref{th:hattheta}, for some 
constant $C_{\gamma}>0$ with probability at least $1-e^{-t},$
\begin{equation}
\label{estim_br}
|\hat b_r-b_r|
\leq 
C_\gamma \frac{\|\Sigma\|_{\infty}^2}{\bar g_r^2}
\biggl(\sqrt{\frac{{\bf r}(\Sigma)}{n}}\bigvee \sqrt{\frac{t}{n}}\bigvee \frac{t}{n}
\biggr)
\sqrt{\frac{t}{n}}.
\end{equation}
and, for all $u\in {\mathbb H},$ 
\begin{equation}
\label{tildetheta}
\Bigl|\langle \tilde \theta_r-\theta_r,u\rangle\Bigr| 
\leq C_{\gamma} \frac{\|\Sigma\|_{\infty}}{\bar g_r}
\sqrt{\frac{t}{n}}\|u\|.
\end{equation}
\end{proposition}

\begin{proof}
It follows from the definition of $\hat b_r$ that
\begin{align}
\label{hatbrbr}
&
\nonumber
|\hat b_r-b_r|= \Bigl|\langle \hat \theta_r, \hat \theta_r'\rangle-(1+b_r)\Bigr|
= 
\\
&
\nonumber
\Bigl|\sqrt{1+b_r}\Bigl\langle \hat \theta_r -\sqrt{1+b_r} \theta_r, \theta_r\Bigr\rangle +
\sqrt{1+b_r}\Bigl\langle \hat \theta_r' -\sqrt{1+b_r}\theta_r, \theta_r\Bigr\rangle
\\
&
\nonumber
+
\Bigl\langle \hat \theta_r -\sqrt{1+b_r} \theta_r, \hat \theta_r' -\sqrt{1+b_r} \theta_r\Bigr\rangle 
\Bigr|
\\
&
\leq 
\Bigl|\sqrt{1+b_r}\Bigl\langle \hat \theta_r -\sqrt{1+b_r} \theta_r, \theta_r\Bigr\rangle\Bigr| +
\Bigl|\sqrt{1+b_r}\Bigl\langle \hat \theta_r' -\sqrt{1+b_r}\theta_r, \theta_r\Bigr\rangle
\Bigr|
\\
&
\nonumber
+
\Bigl|\Bigl\langle \hat \theta_r -\sqrt{1+b_r} \theta_r, \hat \theta_r' -\sqrt{1+b_r} \theta_r\Bigr\rangle\Bigr|.
\end{align}
By bound (\ref{betterrho}), with probability at least $1-e^{-t}$
$$
\Bigl|\Bigl\langle \hat \theta_r -\sqrt{1+b_r} \theta_r,  \theta_r\Bigr\rangle\Bigr|\leq 
C_\gamma \frac{\|\Sigma\|_{\infty}^2}{\bar g_r^2}
\biggl(\sqrt{\frac{{\bf r}(\Sigma)}{n}}\bigvee \sqrt{\frac{t}{n}}\bigvee \frac{t}{n}
\biggr)
\sqrt{\frac{t}{n}}
$$ 
and with the same probability
$$
\Bigl|\Bigl\langle \hat \theta_r' -\sqrt{1+b_r} \theta_r,  \theta_r\Bigr\rangle\Bigr|\leq 
C_\gamma \frac{\|\Sigma\|_{\infty}^2}{\bar g_r^2}
\biggl(\sqrt{\frac{{\bf r}(\Sigma)}{n}}\bigvee \sqrt{\frac{t}{n}}\bigvee \frac{t}{n}
\biggr)
\sqrt{\frac{t}{n}}.
$$ 
By Theorem \ref{th:hattheta}, conditionally on the second sample, 
with probability at least $1-e^{-t}$
\begin{equation}
\label{condit}
\Bigl|\Bigl\langle \hat \theta_r -\sqrt{1+b_r} \theta_r, 
\hat \theta_r'-\sqrt{1+b_r} \theta_r\Bigr\rangle\Bigr|\leq 
C_{\gamma} \frac{\|\Sigma\|_{\infty}}{\bar g_r}
\sqrt{\frac{t}{n}}
\Bigl\|\hat \theta_r'-\sqrt{1+b_r} \theta_r\Bigr\|.
\end{equation}
To bound the norm $\Bigl\|\hat \theta_r'-\sqrt{1+b_r} \theta_r\Bigr\|$ in the right hand side, note that
$$
\Bigl\|\hat \theta_r'-\sqrt{1+b_r} \theta_r\Bigr\|\leq 
\Bigl\|\hat \theta_r'-\theta_r\Bigr\|+ |\sqrt{1+b_r}-1|
\leq \sqrt{2}\|\hat P_r'-P_r\|_{\infty} + \frac{|b_r|}{1+\sqrt{1+b_r}},
$$
where $\hat P_r':=\hat \theta_r'\otimes \hat \theta_r'$ and 
we used the bound 
$$
\|\hat \theta_r'-\theta_r\|^2 = 2-2\langle \hat \theta_r,\theta_r\rangle
\leq 2-2\langle \hat \theta_r,\theta_r\rangle^2 = 2-2\langle \hat P_r',P_r\rangle
=\|\hat P_r'-P_r\|_2^2 \leq 2\|\hat P_r'-P_r\|_{\infty}^2. 
$$
Using bounds (\ref{bd_1}), (\ref{bdbr}) and Theorem \ref{spectrum_sharper},
it is easy to show that with probability at least $1-e^{-t}$
$$
\Bigl\|\hat \theta_r'-\sqrt{1+b_r} \theta_r\Bigr\|\lesssim_{\gamma}
\frac{\|\Sigma\|_{\infty}}{\bar g_r}\biggl(\sqrt{\frac{{\bf r}(\Sigma)}{n}}\bigvee \sqrt{\frac{t}{n}}\biggr).
$$
Together with (\ref{condit}) this implies that, for some $C_{\gamma}>0,$ with probability at least 
$1-2e^{-t}$
$$
\Bigl|\Bigl\langle \hat \theta_r -\sqrt{1+b_r} \theta_r, 
\hat \theta_r'-\sqrt{1+b_r} \theta_r\Bigr\rangle\Bigr|\leq 
C_\gamma \frac{\|\Sigma\|_{\infty}^2}{\bar g_r^2}
\biggl(\sqrt{\frac{{\bf r}(\Sigma)}{n}}\bigvee \sqrt{\frac{t}{n}}\biggr)
\sqrt{\frac{t}{n}}.
$$
It remains to use again bound (\ref{brbd}) on $|b_r|$ and to deduce from (\ref{hatbrbr}) that (\ref{estim_br}) holds with probability at least $1-4e^{-t}.$
To write the probability bound as $1-e^{-t},$ it is enough to adjust 
the constants. 

Under the assumptions of Theorem \ref{th:hattheta}, the proof of bound (\ref{tildetheta}) is straightforward. 

\qed
\end{proof}

We turn now to asymptotic normality of empirical spectral projectors. 
It is easy to see that (\ref{Ehat}), bound (\ref{T_r'''}) on $\|T_r\|_{\infty}$
and Theorem \ref{asymptotic_normality_1} yield the following corollary.

\begin{corollary}
\label{simeicor}
Suppose that 
\begin{equation}
\label{inft'}
\sup_{n\geq 1}\|\Sigma^{(n)}\|_{\infty}<+\infty 
\end{equation}
and 
\begin{equation}
\label{bfr'}
{\bf r}(\Sigma^{(n)})=o(n)\ {\rm as}\ n\goin. 
\end{equation}
Suppose also that 
assumptions \ref{separate_spectrum} and \ref{weak_spectral_measure} of Section \ref{Sec:Asymptotic}  hold. 
Finally, suppose that $\mu^{(n)}=\mu_{r_n}^{(n)}$ is an eigenvalue of $\Sigma^{(n)}$
of multiplicity $1.$ Denote $b^{(n)}=b_{r_n}^{(n)}.$
Then, the finite dimensional distributions of stochastic processes
$$
n^{1/2}\Bigl\langle (\hat P^{(n)}-(1+b^{(n)})P^{(n)})u,v\Bigr\rangle,
u, v\in {\mathbb H}
$$
converge weakly as $n\to\infty$ to the finite dimensional distributions 
of centered Gaussian process $Y(u,v), u,v\in {\mathbb H}$ with covariance 
$\Gamma.$

If, in addition, Assumption \ref{weak_spectral_measure''} holds, then, for all 
$\varphi_n,\psi_n: {\mathbb H}\mapsto {\mathbb H}$ such that $\varphi_n(u)\to u, \psi_n(u)\to u$ as $n\to \infty$ for all $u\in {\mathbb H},$
the finite dimensional distributions of stochastic processes
$$
n^{1/2}\Bigl\langle (\hat P^{(n)}-(1+b^{(n)})P^{(n)})\varphi_n(u),\psi_n(v)\Bigr\rangle,
u, v\in {\mathbb H}
$$
converge weakly as $n\to\infty$ to the same limit.
\end{corollary}

Note that under the assumptions of Corollary \ref{simeicor},
$P^{(n)}=\theta^{(n)}\otimes\theta^{(n)}$
and, with probability tending to $1,$ $\hat P^{(n)}=\hat \theta^{(n)}\otimes \hat \theta^{(n)}$ for eigenvectors $\theta^{(n)}$ of $\Sigma^{(n)}$ and $\hat \theta^{(n)}$ of $\hat \Sigma_n.$ 
We will be able to rephrase the corollary in terms of linear forms of eigenvectors rather than bilinear forms of spectral projectors. 

\begin{theorem}
\label{ANeigen}
Suppose that 
\begin{equation}
\label{inft}
\sup_{n\geq 1}\|\Sigma^{(n)}\|_{\infty}<+\infty 
\end{equation}
and 
\begin{equation}
\label{bfr}
{\bf r}(\Sigma^{(n)})=o(n)\ {\rm as}\ n\goin. 
\end{equation}
Suppose also that 
assumptions \ref{separate_spectrum} and \ref{weak_spectral_measure''} hold and recall that, under
these assumptions, $\theta^{(n)}\to \theta\in {\mathbb H}$
as $n\to\infty.$ 
Finally, assume that the sign of $\hat \theta^{(n)}$ is chosen 
to satisfy the condition $\langle \hat \theta^{(n)},\theta^{(n)}\rangle\geq 0.$
Then, 
the finite dimensional distributions of stochastic processes
$$
n^{1/2}\Bigl\langle \hat \theta^{(n)}-\sqrt{1+b^{(n)}}\theta^{(n)},u\Bigr\rangle,
u\in {\mathbb H}
$$
converge weakly as $n\to\infty$ to the finite dimensional distributions 
of centered Gaussian process  
$
Y(\theta, u), u\in {\mathbb H}. 
$
\end{theorem}

\begin{proof}
Denote  
$$
\rho^{(n)}(u):=\rho_{r_n}^{(n)}(u):=\langle (\hat P^{(n)}-(1+b^{(n)})P^{(n)})\theta^{(n)},u\rangle, u\in {\mathbb H}.
$$
It follows from Corollary \ref{simeicor} and the fact that $Y(\theta,\theta)=0$ 
(see also Theorem Theorem \ref{asymptotic_normality_1} and the definition 
of the process $Y$ and its covariance) that the finite dimensional distributions 
of stochastic processes 
\begin{equation}
\label{comb_comb}
n^{1/2}\Bigl(\rho^{(n)}(u), \rho^{(n)}(\theta^{(n)})\Bigr), u\in {\mathbb H}
\end{equation}
converge weakly
to the Gaussian process $(Y(\theta,u),0), u\in {\mathbb H}.$ 
In particular, this 
implies that 
$$\rho^{(n)}(\theta^{(n)})=O_{\mathbb P}(n^{-1/2})=o_{\mathbb P}(1).$$ 
Under the conditions of the corollary, 
we also have that 
$$
b^{(n)}=O\biggl(\frac{{\bf r}(\Sigma^{(n)})}{n}\biggr)=o(1).
$$
It follows from (\ref{glav}) that 
\begin{align}
&
\nonumber
n^{1/2}\Bigl\langle \hat \theta^{(n)}-\sqrt{1+b^{(n)}}\theta^{(n)},u\Bigr\rangle = 
\frac{n^{1/2}\rho^{(n)}(u)}{\sqrt{1+b^{(n)}+\rho^{(n)}(\theta^{(n)})}}
\\
&
\nonumber
- 
\frac{\sqrt{1+b^{(n)}}}{\sqrt{1+b^{(n)}+\rho^{(n)}(\theta^{(n)})}
\Bigl(\sqrt{1+b^{(n)}+\rho^{(n)}(\theta^{(n)})}+\sqrt{1+b^{(n)}}\Bigr)}
n^{1/2}\rho^{(n)}(\theta^{(n)})\langle \theta^{(n)},u\rangle.
\end{align}
This representation, the convergence of finite dimensional distribution 
of the process (\ref{comb_comb}) and the fact that $\rho^{(n)}(\theta^{(n)})=o_{\mathbb P}(1), b^{(n)}=o(1),$
imply the result.
\qed
\end{proof}

It turns out that the asymptotic normality also holds for the estimator 
with bias correction 
$
\tilde \theta^{(n)}:=\frac{\hat \theta^{(n)}}{\hat b^{(n)}},
$
where $\hat b^{(n)}:=\langle \hat \theta^{(n)}, \hat \theta'^{(n)}\rangle-1,$
$\hat \theta^{(n)},$ $\hat \theta'^{(n)}$ being empirical eigenvectors based on 
the first and on the second subsamples (of size $n$ each) of a sample of size $2n.$ As before, it is assumed that $\langle \hat \theta^{(n)}, \hat \theta'^{(n)}\rangle\geq 0.$
We state the result without proof.

\begin{theorem}
Under assumptions of Theorem \ref{ANeigen}, 
the finite dimensional distributions of stochastic processes
$$
\sqrt{n}\Bigl\langle \tilde \theta^{(n)} - \theta^{(n)},u\Bigr\rangle,
u\in {\mathbb H}
$$
converge weakly to the finite dimensional distributions of  
stochastic process 
$
Y(\theta,u), u\in {\mathbb H}.
$ 
\end{theorem}

Suppose ${\mathbb H}={\mathbb R}^p$ and let $e_1,\dots, e_p$ be an orthonormal 
basis of the space ${\mathbb R}^p.$ For $u\in {\mathbb R}^p,$ let 
$$
\|u\|_{\ell_{\infty}} :=\max_{1\leq j\leq p}|\langle u,e_j\rangle|=
\max_{1\leq j\leq p}|u^{(j)}|.
$$ 
We present now a non-asymptotic bound on 
$\|\tilde\theta_r- \theta_r\|_{\ell_{\infty}}$
that immediately follows from Proposition \ref{prop:tildetheta}. 

\begin{corollary}\label{thm:sup-norm bound}
Suppose all the assumptions of Theorem \ref{th:hattheta} hold and, moreover,
$$
C_{\gamma}\|\Sigma\|_{\infty}\left(\sqrt{\frac{t+\log p}{n}}\bigvee \frac{t+\log p}{n} \right)
\leq \frac{\gamma \bar g_r}{2}.
$$
Then, with probability at least $1-e^{-t},$
$$
\left\|\tilde \theta_r - \theta_r\right\|_{\ell_{\infty}} \leq C_{\gamma}\frac{\|\Sigma\|_{\infty}}{\bar g_r}\sqrt{\frac{t+\log p}{n}}.
$$

\end{corollary}

\noindent\textbf{Example: Eigenvector support recovery.} Our goal is to recover the support of eigenvector $\theta_r$ denoted by 
$$
J_r := {\rm supp}(\theta_r):=\left\lbrace j\,:\,  \theta_r^{(j)} \neq 0 \right\rbrace.
$$ 
It follows from Corollary \ref{thm:sup-norm bound} that a simple hard-thresholding procedure can achieve support recovery. Define
$
\tilde J_r  = \left\lbrace j\,:\,  |\tilde\theta_r^{(j)}| > \beta_n \right\rbrace
$
where $\beta_n := C_{\gamma} \frac{\|\Sigma\|_{\infty}}{\bar g_r}\sqrt{\frac{t+\log p}{n}}$. If
$\rho := \min_{j \in J_r} |\theta_r^{(j)}| > 2\beta_n,$
then we can immediately deduce from Corollary \ref{thm:sup-norm bound} that $\mathbb P\left( \tilde J_r = J_r \right) \geq 1-e^{-t}$. It is well known that the theoretical threshold to perform support recovery in the Gaussian sequence space model is $\beta_n^* \asymp \sigma \sqrt{\frac{t+\log p }{n}}$ where $\sigma$ is the noise variance. The above threshold $\beta_n$ in eigenvector support recovery is similar with the noise variance $\sigma$ replaced by $\frac{\|\Sigma\|_{\infty}}{\bar g_r}$.

\noindent\textbf{Example: Sparse PCA oracle inequality.} We propose a new estimator of $\theta_r$ that satisfies a sparsity oracle inequality with sharp minimax $l_2$-norm rate (see \cite{VuLei2012} for more details about minimax rates in sparse PCA). This estimator is computationally feasible and also adaptive in the sense that no prior knowledge about the sparsity of $\theta_r$ is required. Consider the estimator $\widetilde \theta_r \in \mathbb R^{p}$ obtained by keeping all the components of $\tilde \theta_r$ with their indices in $\tilde J_r$ and setting all the remaining components equal to $0$. We denote by $\|\theta_r\|_{l_0}$ the number of nonzero components of $\theta_r$. Combining the above support recovery property with Corollary \ref{thm:sup-norm bound}, we immediately get the following result.

\begin{theorem}
Let the conditions of Theorem \ref{th:hattheta} be satisfied. Assume in addition that $\rho = \min_{j \in J_r} |\theta_r^{(j)}| \geq 2 \beta_n $. Then, with  probability at least $1-e^{-t}$
\begin{align}
\|\widetilde \theta_r - \theta_r\|_{l_2}^2 \leq C_\gamma^2 \frac{\|\Sigma\|_{\infty}^2}{\bar g_r^2} \|\theta_r\|_{l_0} \frac{t + \log p}{n}.
\end{align}
\end{theorem}

\bibliographystyle{plain}
\bibliography{biblio}

\end{document}